\documentclass[a4paper, 12pt]{article}

\usepackage[utf8]{inputenc}
\usepackage[T1]{fontenc}
\usepackage[a4paper, margin=2.5cm]{geometry}
\usepackage{needspace}
\usepackage{xspace}

\usepackage{libertine} 
\usepackage{inconsolata} 

\usepackage{amsmath, amsthm, amssymb,mathtools}
\usepackage{graphicx}
\usepackage{enumerate}
\usepackage{authblk}
\usepackage{thmtools}
\usepackage{thm-restate}
\usepackage[colorlinks=true, citecolor=red]{hyperref}
\usepackage[capitalize]{cleveref}
\usepackage{float}
\usepackage{faktor}
\usepackage{subcaption}

\declaretheorem[name=Theorem, numberwithin=section]{theorem}
\declaretheorem[name=Lemma, sibling=theorem]{lemma}
\declaretheorem[name=Proposition, sibling=theorem]{proposition}

\declaretheorem[name=Corollary, sibling=theorem]{corollary}
\declaretheorem[name=Conjecture, sibling=theorem]{conjecture}

\declaretheorem[name=Claim, sibling=theorem]{claim}

\declaretheorem[name=Remark, style=remark, sibling=theorem]{remark}

\def\cqedsymbol{\ifmmode$\lrcorner$\else{\unskip\nobreak\hfil
\penalty50\hskip1em\null\nobreak\hfil$\lrcorner$
\parfillskip=0pt\finalhyphendemerits=0\endgraf}\fi}

\crefname{subsection}{Subsection}{Subsections}
\crefname{claim}{Claim}{Claims}
\crefname{problem}{Problem}{Problems}
\crefname{figure}{Figure}{Figures}

\def\Aut{\operatorname{Aut}}

\def\Shear{\operatorname{Shear}}

\newcommand{\RG}{$\{R,G\}$\xspace}  
\newcommand{\GB}{$\{G,B\}$\xspace}  

\let\le\leqslant
\let\ge\geqslant

\def\bR{\mathbb{R}}
\def\bZ{\mathbb{Z}}

\def\cH{\mathcal{H}}
\def\cQ{\mathcal{Q}}

\def\cT{\mathcal{T}}

\def\ST{\operatorname{ST}}

\def\SL{\operatorname{SL}}
\def\Sp{\operatorname{Sp}}

\title{Reconfiguration of square-tiled surfaces}

\author[1]{Vincent Delecroix}
\author[1,2]{Clément Legrand-Duchesne}
\affil[1]{CNRS, LaBRI, Université de Bordeaux, Bordeaux, France.}
\affil[2]{Theoretical Computer Science Department, Faculty of Mathematics and Computer Science, Jagiellonian University, Kraków, Poland.}
\date{\today}

\begin{document}

\maketitle

\begin{abstract}
  We consider a combinatorial reconfiguration problem on a subclass of
  quadrangulations of surfaces called square-tiled surfaces. Our elementary move
  is a shear in a cylinder that corresponds to a well-chosen sequence of
  diagonal flips that preserves the square-tiled properties. We conjecture that
  the connected components of this reconfiguration problem are in bijection with
  the connected components of the moduli space of quadratic differentials. We
  prove that the conjecture holds in the so-called hyperelliptic components of
  Abelian square-tiled surfaces. More precisely, we show that any two such
  square-tiled surfaces of genus $g$ can be connected by $O(g)$ powers of
  cylinder shears.
\end{abstract}

\section{Introduction}\label{sec:intro}

In this article, we consider the combinatorial reconfiguration problem given by
the action of cylinder shears on the set of square-tiled surfaces. We review in
\cref{ssec:combinatorialReconfiguration} combinatorial reconfiguration
that aims to study reachability or connectivity of state spaces from either a
geometric or computational perspective. We next turn in
\cref{ssec:flipGraph} to a classical combinatorial reconfiguration of
geometric origin: diagonal flips in graphs embedded in surfaces.  The latter is
relevant as a square-tiled surface is a special case of quadrangulation and a
cylinder shear can be decomposed as a sequence of diagonal flips. In
\cref{ssec:introSquareTiledSurfaces}, we define cylinder shears on
square-tiled surfaces and state a conjectural classification of connected
components for cylinder shears on square-tiled surfaces
(\cref{conj:squareTiledSurfacesConnectedComponents}). We then state
our main result which settles some special cases of this conjecture
(\cref{thm:connectednessHyperelliptic}).

\subsection{Combinatorial reconfiguration}
\label{ssec:combinatorialReconfiguration}
Combinatorial reconfiguration aims to study the action of the monoid generated
by a set of (possibly partial) functions $\{f_i: X \to X\}$, called the
\emph{(combinatorial) moves} on some finite set $X$ called the \emph{state
  space}. This situation can conveniently be encoded in the
\emph{reconfiguration graph} (which is an oriented multi-graph) whose set of
vertices is $X$ and whose set of edges is $\{(x, f_i(x)): x \in X, i \in
I\}$. Most often, $X$ is a "complicated" set and the moves are "elementary" in
the sense that $x$ does not "differ" much from $f_i(x)$.

Famous reconfiguration problems include games such as the Rubik's Cube or
Sokoban for example, or problems on graphs such as reconfiguration of
independent sets or colourings (see the comprehensive surveys
\cite{van2013complexity,Nishimura2018,mynhardt2019reconfiguration}).

One crucial question in combinatorial reconfiguration is the connectedness of
the reconfiguration graph (equivalently the transitivity of the action). More
precisely, one would like to find simple invariants or efficient procedures to
decide if two elements of the state space are equivalent, that is, belong to the
same connected component. A natural refinement of this question consists in
studying the geometry of the connected components, e.g. their diameter.

When the combinatorial moves are endowed with probabilities, this turns the
action into a Markov chain (see \cite{levin2017markov} for an introduction to
Markov chains and their analysis). Under favourable conditions, the chain is
irreducible if the action is transitive and its stationary measure is the
uniform measure on $X$. A necessary condition for this to happen is the
connectedness of the reconfiguration graph. In such case, a simulation of the
Markov chain gives access to a random generator of elements in $X$. This goes
under the name of \emph{Monte-Carlo} generator. The quality of this generator
which is generally not uniform after a finite execution of the Markov chain is
quantified by the so-called mixing time. This random generation turns out to be
one important motivation for combinatorial reconfiguration both from a
theoretical and an experimental perspective.

\subsection{Edge flip in polygons and embedded graphs}
\label{ssec:flipGraph}
Consider a convex $n$-gon and its triangulations. The \emph{flip} operation
consists in flipping an edge: remove an edge of the triangulation and replace it
with the other diagonal in the created quadrilateral. In that context, the
reconfiguration graph $\cT_n$ is usually called a \emph{flip graph}.  Sleator,
Tarjan and Thurston~\cite{SleatorTarjanThurston-1988}, completed by the work of
Pournin~\cite{Pournin2014}, proved that the diameter of the flip graph $\cT_n$
is equal to $2n - 10$ for $n \ge 13$. This is one of the very few examples where
the value of the diameter of the reconfiguration graph is exactly known. Molloy,
Reed and Steige \cite{molloy1997Mixing} proved that the associated Markov chain,
called the \emph{flip dynamics}, mixes rapidly. Since then, the mixing time of
the flip dynamics has been improved several times, although the exact rate is
not known and has been conjectured by Aldous~\cite{Aldous1994} to be
$O(n^{3/2})$, up to logarithmic factors. The current best upper bound of
$O(n^3\log^3(n))$ is due to Eppstein and Frishbreg~\cite{eppstein2023Improved}.

Let us now move from polygons to graphs embedded in surfaces. Let $e$ be an edge
adjacent to two distinct faces of degree $d_1$ and $d_2$ in an embedded graph. A
\emph{flip} of $e$ consists in first removing $e$ from the graph (see
\cref{fig:flip}). This merges the two adjacent faces into one of degree
$d_1 + d_2 - 2$. Then, we reintroduce an edge $e'$ inside the newly created
face. Note that among the flips associated to $e$ there are some that preserve
the faces' degrees.

\begin{figure}[!h]
\centering
\includegraphics{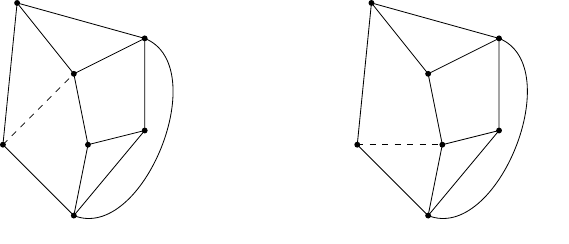}
\caption{Two combinatorial maps that differ by one flip. The flipped edge is
 dashed. This flip preserves the degree of the faces.}
\label{fig:flip}
\end{figure}

Let $\cT_{g,n}$ and $\cQ_{g,n}$ be respectively the set of isomorphism classes
of triangulations and of quadrangulations of genus $g$ with $n$ vertices. One
can turn $\cT_{g,n}$ and $\cQ_{g,n}$ into reconfiguration graphs by putting an
edge between $m$ and $m'$ if $m$ and $m'$ are related by a flip.  The
connectedness of $\cT_{g,n}$ has been proved by various geometric and
combinatorial methods, see~\cite{Mosher1988} and the references
therein. Estimates on the diameter have only been settled recently in the most
general situation.
\begin{theorem}[Disarlo-Parlier~\cite{DisarloParlier2019}]
\label{thm:DisarloParlierTriangulationsAreConnected}
The diameter of $\cT_{g,n}$ is $\Theta(g \log (g+ 1) + n)$.
\end{theorem}
Let us mention that the lower bound on the diameter follows from the techniques
in~\cite{SleatorTarjanThurston-1992}. The case of the sphere ($g=0$) has
received more attention and more precise asymptotics are known,
see~\cite{Pournin2014,Frati2017,ParlierPournin2017,CardinalHoffmanKustersTothWettstein2018,ParlierPournin2018a,ParlierPournin2018b}.

Finally, we consider some results concerning the mixing time in the case of the
sphere. We denote $\cT^{root}_{g,n}$ and $\cQ^{root}_{g,n}$ the triangulations
and quadrangulations of genus $g$ with $n$ vertices and one rooted edge, that is
distinguished and oriented\footnote{Note that our definitions of
  $\cT^{root}_{g,n}$ and $\cQ^{root}_{g,n}$ differ
  from~\cite{CaraceniStauffer2020,Budzinski2017}, in which $n$ denotes the
  number of faces. However, by Euler's formula, the number of faces is
  proportional to the number of vertices, up to an additive term which is linear
  in the genus.  Hence, their results extend to our definitions of
  $\cT^{root}_{g,n}$ and $\cQ^{root}_{g,n}$.}.
\begin{theorem}[\cite{Budzinski2017,CaraceniStauffer2020}]
  The flip dynamics on $\cT^{root}_{0,n}$ has mixing time $\Omega(n^{5/4})$.
  The flip dynamics on $\cQ^{root}_{0,n}$ has mixing time $\Omega(n^{5/4})$ and
  $O(n^{13/2})$.
\end{theorem}

\subsection{Cylinder shear in square-tiled surfaces}
\label{ssec:introSquareTiledSurfaces}
A square-tiled surface is a quadrangulation of a connected surface with a
2-colouring of its edges such that all faces have an alternating boundary. One
obtains a square-tiled surface geometrically by taking copies of a unit square
and gluing horizontal edges to horizontal edges and vertical edges to vertical
edges. We refer the reader to \cref{sssec:cubicEdgeTricoloredCubicGraphs} for
more detailed definitions. Square-tiled surfaces were introduced in the context
of Abelian and quadratic differentials on Riemann surfaces by
Zorich~\cite{Zorich2002}. In this article, we will actually consider a slight
generalization of square-tiled surfaces where we allow \emph{folded edges} which
turns out to be crucial in our connectedness proof (see
\cref{sssec:cubicEdgeTricoloredCubicGraphs} for a precise definition).

In general, a single diagonal flip operation does not preserve the square-tiled
property of the surface. Indeed, diagonal flips change the parity of the degree
of the vertices incident to the flipped edge (see \cref{fig:diagonalFlip}), but
the 2-colouring of the edges of the square-tiled surfaces imposes the vertices
to have even degrees.
\begin{figure}[!h]
\centering
\includegraphics{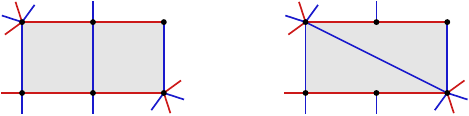}
\caption{Diagonal flips do not preserve square-tiled surfaces.}
\label{fig:diagonalFlip}
\end{figure}

A \emph{horizontal cylinder} in a square-tiled surface is a cycle of squares
adjacent along their vertical edges. A \emph{horizontal cylinder shear} consists
in doing all flips on the vertical edges in the cylinder (see
\cref{fig:cylinderShear}). Note that these flips commute. One defines similarly
\emph{vertical cylinders} and their shears. Cylinder shears hence preserve the
set of square-tiled surfaces and the number of squares. We will see in
\cref{sssec:shearDef}, that cylinder shears present unexpected similarities with
another reconfiguration operation: Kempe changes on graph colourings.
\begin{figure}[!h]
\centering
\includegraphics{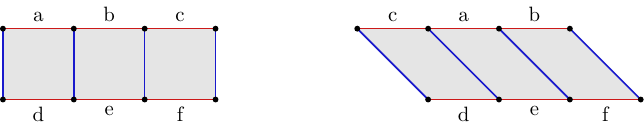}
\caption{A horizontal cylinder shear.}
\label{fig:cylinderShear}
\end{figure}

Beyond the number of squares, it turns out that the degree of the vertices is
also invariant under cylinder shears. The \emph{profile} of a square-tiled
surface is a pair $(\mu,k)$ made of an integer partition
$\mu = [1^{\mu_1},2^{\mu_2},\ldots]$ where $\mu_i$ is the number of vertices of
degree $2 \mu_i$ and $k$ is the number of folded edges. Euler's formula relates
the profile $(\mu, k)$ to the genus $g$ of the underlying surface as follows,
\begin{equation}
\label{eq:EulerFormula}
\sum_{i \ge 1} \mu_i (i - 2) = 4 g - 4 + k.
\end{equation}
We let $\ST(\mu, k)$ denote the set of square-tiled surfaces with given profile
$(\mu, k)$ and by abuse of notation $\ST(\mu)$ when $k=0$.

It turns out that $\ST(\mu, k)$ might still be non-connected under cylinder
shears. We introduce a first invariant of the connected components. Let $\tau$
be a square-tiled surface. We say that $\tau$ is \emph{Abelian} if one can build
this surface in such way that the top sides are only glued to the bottom sides,
and the right sides to the left sides. We say that $\tau$ is \emph{quadratic}
otherwise. Let $\ST_{Ab}(\mu, k)$ and $\ST_{quad}(\mu, k)$ be respectively the
set of Abelian and quadratic square-tiled surfaces in $\ST(\mu, k)$. In
\cref{fig:AbQuadSameStratum} we show an Abelian and a quadratic square-tiled
surface in the stratum $\ST([4^2], 0)$. It is easy to see that the profile
$(\mu,k)$ of an Abelian square-tiled surface is such that $k=0$ and the entries
in $\mu$ are even, that is $\mu_{2i+1}=0$ for all $i\ge 0$.

\begin{figure}[!ht]
  \centering
  \begin{subfigure}{.4\textwidth}
    \centering
    \includegraphics{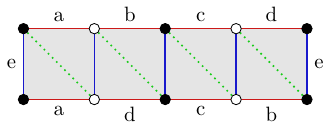}
    \caption{An Abelian square-tiled surface in $\ST_{Ab}([4^2])$ (stratum
      $\cH(1^2)$).}
    \label{sfig:AbSquareTiled}
  \end{subfigure}%
  \hspace{.1\textwidth}
  \begin{subfigure}{.4\textwidth}
    \centering
    \includegraphics{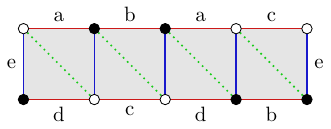}
    \caption{A quadratic square-tiled surface in $\ST_{quad}([4^2])$ (stratum $\cQ([2^2])$).}
    \label{sfig:quadSquareTiled}
  \end{subfigure}
  \caption{Two square-tiled surfaces in $\ST([4^2])$ (genus $g=2$) that are not
    connected under cylinder shears.}
  \label{fig:AbQuadSameStratum}
\end{figure}

The names Abelian and quadratic are borrowed from the theory of Riemann surfaces
and their Abelian and quadratic differentials. It turns out that our main
conjecture for the connectedness under cylinder shears
(\cref{conj:squareTiledSurfacesConnectedComponents} below) is intimately linked
to them. We introduce the necessary background now and refer the reader to
\cref{sec:background} for more details. To each square-tiled surface, one can
associate a Riemann surface endowed with a quadratic differential. It turns out
that $\ST_{Ab}(\mu)$ (respectively $\ST_{quad}(\mu,k)$) corresponds to the
quadratic differentials that are the squares of an Abelian differential
(resp. not the square of an Abelian differential). The moduli space of the
Abelian differentials and the quadratic differentials that are not the square of
an Abelian one are stratified according to the degree of their zeros. The
stratum of the Abelian differentials associated to the square-tiled surfaces in
$\ST_{Ab}(\mu)$ is denoted $\cH(1^{\kappa_1}, 2^{\kappa_6}, \ldots)$, where
$\kappa_i = \mu_{2i+2}$, while the one associated to $\ST_{quad}(\mu,k)$ is
denoted $\cQ(-1^{\kappa_{-1}}, 1^{\kappa_1}, 2^{\kappa_2}\ldots)$, where
$\kappa_{-1}= \mu_1+k$ and for $i\ge 1, \kappa_i = \mu_{i+2}$. Let us emphasise
that $\mu_2$ does not play any role in the associated stratum. The strata are
not necessarily connected but their connected components have been classified by
Kontsevich and Zorich~\cite{KontsevichZorich2003} and by
Lanneau~\cite{Lanneau2008}. As a consequence of this classification we have
that:
\begin{theorem}[\cite{KontsevichZorich2003,Lanneau2008}]
  \label{thm:strataConnectedComponentNumber}
  \begin{enumerate}
  \item[]
  \item Each Abelian stratum $\cH(\kappa)$ has at most three connected components.
  \item Each quadratic stratum $\cQ(\kappa)$ has at most two connected components.
  \item In genus $0$, that is
    $\displaystyle \sum_{i \ge -1} i \cdot \kappa_i = -4$, $\cQ(\kappa)$ is
    non-empty and connected.
  \end{enumerate}
\end{theorem}
In this article we will focus on the genus $0$ case and some specific connected
components of $\cH(\kappa)$ in higher genera that are called \emph{hyperelliptic
  components} (see \cref{sec:background} for the definition). For now, let us
just mention that for $\mu=[2^{\mu_2}, 4g-2]$ or $\mu=[2^{\mu_2},(2g)^{2}]$
where $g \ge 2$ and $\mu_2 \ge 0$, there is a subset $\ST^{hyp}_{Ab}(\mu)$ of
$\ST_{Ab}(\mu)$, closed under cylinder shears, which corresponds to the
hyperelliptic connected component of $\cH(2g-2)$ or $\cH((g-1)^2)$ respectively.

\medskip We conjecture that the connected components of the moduli space
partition the square-tiled surfaces into equivalence classes for the shearing
operation:
\begin{restatable}[Generalizing Conjecture~4
  in~\cite{BuchinEtAl}]{conjecture}{conjquadrangulations}\label{conj:squareTiledSurfacesConnectedComponents}
  Let $\mu$ be an integer partition of $n$ and $k$ a non-negative integer such
  that $$\sum_{i\ge 1} \mu_i(i -2) = 4g-4+k.$$ Let $S$ and $S'$ be two
  square-tiled surfaces in $\ST(\mu,k)$. Then $S$ and $S'$ are equivalent via
  cylinder shears if and only if they belong to the same connected component of
  the moduli space of quadratic differentials.
\end{restatable}
According to \cref{thm:strataConnectedComponentNumber}, each $\ST(\mu,k)$ would
then be made of at most 5 connected components under cylinder shears.
\cref{conj:squareTiledSurfacesConnectedComponents} is a direct generalization of
Conjecture~4 in~\cite{BuchinEtAl}, which was supported by numerical evidence, to
square-tiled surfaces with folded edges.

It is relatively straightforward to prove that being in the same connected
component of the moduli space is a necessary condition:
\begin{proposition} \label{prop:connectedComponentsAndConnectedComponents}
Let $S$ be a square-tiled surface and $S'$ obtained from $S$ by a cylinder
shear. Then the quadratic differentials associated to $S$ and $S'$ belong
to the same connected component of the strata of the moduli space of quadratic
differentials.
\end{proposition}
The underlying reason of \cref{prop:connectedComponentsAndConnectedComponents}
is that a cylinder shear can be realised in the moduli space as a continuous
motion. Our main contribution is the two following theorems
that provide a partial answer to
\cref{conj:squareTiledSurfacesConnectedComponents}.

\begin{theorem}\label{thm:connectednessHyperelliptic}
  Let $g \ge 1$ and $\mu=[2^{\mu_2}, 4g-2]$ or $\mu=[2^{\mu_2}, (2g)^2]$. Then
  any two square-tiled surfaces in $\ST_{Ab}^{hyp}(\mu)$ are connected by
  a sequence of at most $\Theta(g)$ powers of cylinder shears.
\end{theorem}

\begin{theorem}\label{thm:connectednessSphere}
  Let $\mu$ be an integer partition and $k$ a positive integer such that
  $\sum \mu_i(i - 2) = k-4$. Then $\ST(\mu,k)$ is non-empty. Moreover, if
  $\mu_1 \le 1$, then $ST(\mu,k)$ is connected and has
  diameter $\Theta(k)$ with respect to powers of cylinder shears.
\end{theorem}

\cref{thm:connectednessSphere} generalises a result of Cassaigne, Ferenczi and
Zamboni~\cite{CassaigneFerencziZamboni2011} that proves the case of
$\ST([k-2],k)$. Some natural extensions of our results would be to prove
\cref{conj:squareTiledSurfacesConnectedComponents} for the quadratic
hyperelliptic components or for generic profiles $(\mu,k)$ in genus 0, that is
without restrictions on $\mu_1$.

As we already mentioned, the connectedness of $\ST([k-2],k)$ under cylinder
shears was considered in~\cite{CassaigneFerencziZamboni2011}. One of their
motivation was the theory of dynamical systems and more precisely, the dynamics
of interval exchange transformations, see~\cite{FerencziZamboni2010}. Note that
beyond connectedness, \cite{CassaigneFerencziZamboni2011} provides an explicit
formula for the cardinality of $\ST([k-2],k)$.

\subsection{Organisation and sketch of proof}
We start by giving all definitions related to square-tiled surfaces and shears
in \cref{sec:background}. This section also contains background on Abelian and
quadratic differentials and the definition of the weighted stable graph of a
square-tiled surface, that encodes the structure of its cylinders. In
\cref{sec:connectingtoPathLike}, we focus on the square-tiled surfaces of genus
0 and show how they can be connected to some path-like square-tiled surface,
which contains only one horizontal cylinder. We derive from the existence of
these path-like square-tiled surfaces the lower bound of
\cref{thm:connectednessSphere}. Then in \cref{ssec:reconfigurationPathLike}, we
conclude the proof of~\cref{thm:connectednessSphere} by showing that all these
path-like configurations are related by cylinder shears. To do so, we show that
they are all equivalent to a canonical one, by progressively introducing
structure in the path-like configuration. Finally, we address
in~\cref{sec:hyperellipticSquareTiledSurfaces} the case of Abelian hyperelliptic
square-tiled surfaces by reducing to square-tiled surfaces of genus zero,
thereby proving~\cref{thm:connectednessHyperelliptic}.

To conclude, we discuss in \cref{sec:discussion} several works related to our
result as well as possible extensions.
\subsection*{Acknowledgments}
The second author was partially supported by the Polish National Science Centre
under grant no. 2019/34/E/ST6/00443.\\
Both authors benefited from the ANR project Modiff ANR-19-CE40-0003-01. \\
Both authors thank Luke Jeffreys for organizing the workshop
\textit{Square-tiled surfaces: a classification of cylinder block shearing
orbits} at the Heilbronn Institute in Bristol in summer 2023 during which we
had the opportunity to discuss this work
(see \url{https://heilbronn.ac.uk/2023/06/12/frg-square-tiled-surfaces/}).

\section{Square-tiled surfaces}
\label{sec:background}

\subsection{Square-tiled surfaces and their dual tricoloured cubic graphs}
\label{sssec:cubicEdgeTricoloredCubicGraphs}
Recall that a square-tiled surface is a quadrangulation of a connected surface
whose edges are assigned a 2-colouring corresponding to the horizontal and
vertical directions.

Let us introduce a natural encoding of a square-tiled surface that we will use
as our main combinatorial definition. We fix an arbitrary labelling of the
bottom-left and top-right corners of each square by the elements of
$[n]\coloneq \{1, \dots n\}$ where $n$ is twice the number of squares. We define
$\tau_R$, $\tau_G$, $\tau_B$ three involutions without fixed points on $[n]$,
such that $\tau_R$, $\tau_G$ and $\tau_B$ record respectively the adjacencies of
the labels across the horizontal edges, inside a square and across the vertical
edges. By drawing diagonals in each square we turn the square-tiled surface
$\tau$ into a triangulation. In this triangulation, the involution $\tau_G$
encodes adjacencies along the diagonal. We always use the following colour
convention for the edges (see \cref{fig:AbQuadSameStratumLabelled}):
\begin{itemize}
\item horizontal edges (ie the transpositions in the cycle decomposition of $\tau_R$) are red,
\item diagonal edges (ie the transpositions of $\tau_G$) are green,
\item vertical edges (ie the transpositions of $\tau_B$) are blue.
\end{itemize}

\begin{figure}[!h]
\centering
\begin{subfigure}{.4\textwidth}
\includegraphics{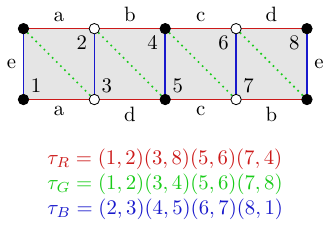}
\caption{A labelling of the square-tiled surface from \cref{sfig:AbSquareTiled}.}
\end{subfigure}%
\hspace{.1\textwidth}
\begin{subfigure}{.4\textwidth}
\includegraphics{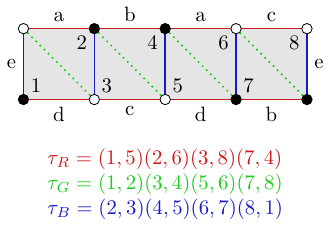}
\caption{A labelling of the square-tiled surface from \cref{sfig:quadSquareTiled}.}
\end{subfigure}
\caption{Labellings of the two square-tiled surfaces from \cref{fig:AbQuadSameStratum}.}
\label{fig:AbQuadSameStratumLabelled}
\end{figure}

Conversely, starting from a triple of involutions without fixed points
$\tau = (\tau_R, \tau_G, \tau_B)$ such that they act transitively on $[n]$, one
can build a square-tiled surface. If the action is not transitive, then the
resulting surface is not connected.

We now introduce folded edges. Instead of using bicoloured squares as building
blocks we use tricoloured triangles (that should be thought as ``half'' of a
square cut along its diagonal). We pick $n$ copies of them that we label from
$1$ to $n$. Then we consider three involutions (possibly with fixed points)
$\tau_R$, $\tau_G$ and $\tau_B$ that determine the gluings of the edges with
colours respectively $R$, $G$ and $B$. If $j$ is a fixed point of $\tau_i$ then
the edge coloured $i$ on the $j$-th triangle is glued to itself by a 180-degree
rotation (see \cref{fig:squareTiledSurfacesWithFoldedEdges}). From now on, we
will use the term square-tiled surface to refer to a triple of involutions
(possibly with fixed points) that act transitively on $[n]$ and we will denote
$S(\tau)$ the tricoloured triangulation we have just constructed (see
\cref{fig:squareTiledSurfacesWithFoldedEdgesS}). We call $n$ the \emph{number of
  triangles} in $S(\tau)$. A \emph{red half-edge} (respectively \emph{green
  half-edge} and \emph{blue half-edge}) is a fixed point of $\tau_R$
(resp. $\tau_G$ and $\tau_B$).

We denote $S^*(\tau)$ the \emph{tricoloured cubic graph} dual to the
triangulation $S(\tau)$ (see
\cref{fig:squareTiledSurfacesWithFoldedEdgesDual}). The graph $S^*(\tau)$ has
vertex set $\{1, \ldots, n\}$. There is a red, green or blue edge between $i$
and $j$ if respectively $(i, j)$ is a transposition in the cycle decomposition
of $\tau_R$, $\tau_G$ or $\tau_B$. And there is a red, green or blue half-edge
at $i$ if it is a fixed point of respectively $\tau_R$, $\tau_G$ or
$\tau_B$. The graph $S^*(\tau)$ is naturally embedded in the square-tiled
surface $S(\tau)$ and at every vertex the cyclic order colours of the three
adjacent edges or half-edges are counter-clockwise in the order red then green
then blue.

We call \emph{$i$-faces} of $S^*(\tau)$ or $3i$-gons the faces of $S^*(\tau)$ of
degree $3i$. For example, there is one 2-face, or also called hexagon, and one
3-face in \cref{fig:squareTiledSurfacesWithFoldedEdgesDual}. Note that the
$i$-faces are dual to the vertices of degree $3i$ in $S(\tau)$.

\begin{figure}[!h]
\begin{subfigure}[t]{.48\textwidth}
  \centering
  \includegraphics{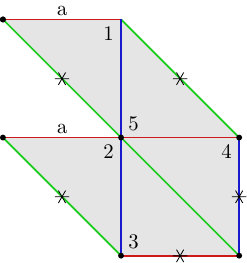}
  \caption{The triangulation $S(\tau)$ encoded by $\tau$}
  \label{fig:squareTiledSurfacesWithFoldedEdgesS}
\end{subfigure}
\hfill
\begin{subfigure}[t]{.48\textwidth}
  \centering
  \includegraphics{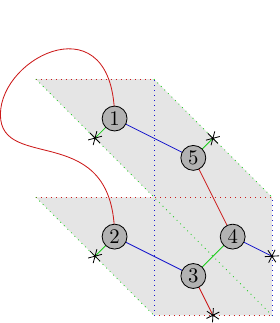}
  \caption{Its dual tricoloured cubic graph $S^*(\tau)$}
  \label{fig:squareTiledSurfacesWithFoldedEdgesDual}
\end{subfigure}
\caption{A square-tiled surface $\tau$ represented by its tricoloured
  triangulation and its tricoloured cubic graph. We have $\tau_R = (1,2)$,
  $\tau_G=(3,4)$ and $\tau_B=(1,5)(2,3)$. In particular, $\tau$ is of genus $0$
  and has profile $([2,3],5)$.  When drawing the tricoloured triangulation of a
  square-tiled surface $\tau$, we will mark with a star the edges glued to
  themselves, that is, the fix points of $\tau$. Likewise, we will represent the
  half-edges in $S^*(\tau)$ with a star at their endpoint.}
\label{fig:squareTiledSurfacesWithFoldedEdges}
\end{figure}

Given a triangulation $S(\tau)$ with $n$ triangles and $\sigma \in S_n$ we
denote
$\tau^\sigma = (\sigma \tau_R \sigma^{-1}, \sigma \tau_G \sigma^{-1},\allowbreak
\sigma \tau_B \sigma^{-1})$. The triangulated square-tiled surface
$S(\tau^\sigma)$ corresponds to the same underlying triangulation as $S(\tau)$
but where the labels of the triangles have been permuted by $\sigma$. We say
that two square-tiled surfaces $\tau$ and $\tau'$ are \emph{isomorphic} if they
have the same number of triangles $n$ and there exists a permutation
$\sigma \in S_n$ such that $\tau' = \tau^\sigma$, in other words, $S(\tau)$ is
isomorphic to $S(\tau')$.

The \emph{profile} of $\tau$ is the pair $(\mu, k)$, where $\mu$ is the integer
partition counting the $i$-faces for any $i$, while $k$ is the total number of
half-edges. We use the standard notation $\mu = [1^{\mu_1}, 2^{\mu_2}, \ldots]$
to denote $1$ repeated $\mu_1$ times, $2$ repeated $\mu_2$ times, etc. We denote
$\ST(\mu,k)$ the set of isomorphism classes of square-tiled surfaces with
profile $(\mu, k)$. Let us note that Euler's characteristic allows to compute
the genus $g$ of a square-tiled surface from its profile $(\mu, k)$ by
\begin{equation} \label{eq:squareTiledSurfaceEulerCharacteristic}
4g - 4 = \sum_i \mu_i \cdot (i - 2) - k.
\end{equation}
Let us also note that the number of triangles $n$ satisfies
\begin{equation}
\label{eq:triangleNumberAndProfile}
n = \sum_{i \ge 1} \mu_i \cdot i.
\end{equation}

A square-tiled surface $\tau$ is \emph{Abelian} if there exists a non-trivial
partition $A_+ \sqcup A_- = [n]$ such that $\tau_i(A_+) = A_-$ for each
$i \in \{R,G,B\}$. Equivalently, it is Abelian if it has no half-edges and its
dual graph $S^*(\tau)$ is bipartite. A square-tiled surface which is not Abelian
is called \emph{quadratic}. We denote $\ST_{Ab}(\mu,k)$ and $\ST_{quad}(\mu,k)$
respectively the Abelian and quadratic square-tiled surfaces of profile
$(\mu,k)$. By definition $\ST(\mu,k)$ is the disjoint union of $\ST_{Ab}(\mu,k)$
and $\ST_{quad}(\mu,k)$.

We finish this introduction to square tiled-surfaces by giving a parity
restriction on the profile.
\begin{lemma} \label{lem:parityCondition}
Let $\tau$ be a square-tiled surface on $[n]$ with profile $(\mu, k)$. Then
we have
\begin{equation}
\label{eq:parityCondition}
3 n - k \equiv 2 \sum_{i \ge 1} \mu_{2i} \mod 4.
\end{equation}
\end{lemma}

\begin{proof}
  The faces of $S^*(\tau)$ are in bijection with the cycles of the product
  $\tau_R \tau_G \tau_B$ whose cycle type is $\mu$. We claim
  that~\eqref{eq:parityCondition} follows from the signature morphism
  $\epsilon: (S_n,\cdot) \to (\mathbb{Z} / 2\mathbb{Z},+)$ which gives
  $\epsilon(\tau_R) + \epsilon(\tau_G) + \epsilon(\tau_B) = \epsilon(\mu)$.
  Indeed $\epsilon(\tau_i) = \frac{n - k_i}{2} \mod 2$ where $k_i$ is the number
  of fixed points of $\tau_i$, because $\tau_i$ is an involution. Hence
  $\epsilon(\tau_R) + \epsilon(\tau_G) + \epsilon(\tau_B) = \frac{3n-k}{2} \mod
  2$. And we have $\epsilon(\mu) = \sum_{i \ge 1} \mu_{2i} \mod 2$ which
  concludes the proof of the lemma.
\end{proof}

\subsection{Cylinder shear}\label{ssec:shearDef}\label{sssec:shearDef}
We now introduce formally the cylinder shears on tricoloured cubic graphs
$S^*(\tau)$. They generalise the induction defined in the case of trees
in~\cite{CassaigneFerencziZamboni2011}.

Let $\tau = (\tau_R, \tau_G, \tau_B)$ be a square-tiled surface. Let us
consider a pair of colours $\{i,j\} \subset \{R,G,B\}$. A
\emph{$\{i,j\}$-component} is a connected component of the subgraph of
$S^*(\tau)$ in which we only keep the edges coloured $i$ or $j$. More
combinatorially, these correspond to the orbits of the subgroup
$\langle \tau_i, \tau_j \rangle$. Because the graph induced on $\{i,j\}$ has
degree 2 and possibly half-edges, each $\{i,j\}$-component is either a path or a
cycle and we call them respectively \emph{$\{i,j\}$-path} or
\emph{$\{i,j\}$-cycle}.

Let $c \subset [n]$ be a \RG-component. Let $c_R \coloneq \tau_R|_c$ and
$c_G \coloneq \tau_G|_c$ be the restrictions of $\tau_R$ and $\tau_G$ to $c$ (by
definition $\tau_R$ and $\tau_G$ preserve $c$). The \emph{$(R,G)$-shear along
  $c$} is the square-tiled surface
$\Shear_{c,R,G}(\tau) \coloneq (\tau_R \cdot c_R \cdot c_G, \tau_G \cdot c_R
\cdot c_G, \tau_B)$.

More graphically, the $(R,G)$-shear along $c$ affects $S^*(\tau)$ by
switching colours $R$ and $G$ in $c$ and sliding all adjacent edges coloured
$B$, see \cref{fig:rotation}. Note that $\tau_B$ remains unchanged under this
operation.
\begin{figure}[!ht]
\begin{center}%
\begin{subfigure}{.9\linewidth}
\begin{center}\includegraphics{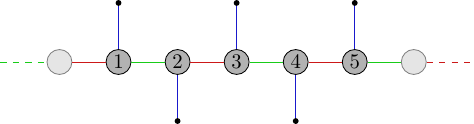}\end{center}
\caption{A piece of a \RG-component of a square-tiled surface $\tau$
  (represented by its dual tricoloured cubic graph) and its neighbourhood.}
\label{fig:originalRGComponent}
\end{subfigure}
\begin{subfigure}{0.4\linewidth}
\includegraphics{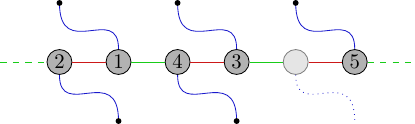}
\caption{A piece of the square-tiled $\Shear_{c,R,G}(\tau)$ with the same vertex position as in \cref{fig:originalRGComponent}.}
\end{subfigure}
\hspace{0.08\linewidth}
\begin{subfigure}{0.4\linewidth}
\includegraphics{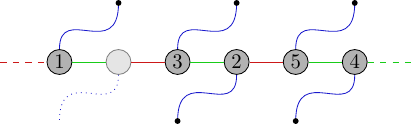}
\caption{A piece of the square-tiled $\Shear_{c,G,R}(\tau)$ with the same vertex position as in \cref{fig:originalRGComponent}.}
\end{subfigure}
\end{center}
\caption{A piece of a \RG-component and its neighbourhood before and after a $(R,G)$-shear.}
\label{fig:rotation}
\end{figure}

Let $\tau' = \Shear_{c,R,G}(\tau)$. The vertices of $S^*(\tau)$ and
$S^*(\tau')$ are $[n]$ and are hence in natural bijection. Moreover, there is a
natural bijection between the edges of $S^*(\tau)$ and $S^*(\tau')$ that
preserves the colours of all edges but those in the component $c$, which are
exchanged.  As we already mentioned, blue edges remain unchanged. Similarly,
green and red edges outside of $c$ remain unchanged. We only need to discuss how
to match the (red and green) edges of $S^*(\tau)$ and $S^*(\tau')$ inside $c$.
Recall that on $c$ we have
\[
\tau'_R|_c = \tau_G|_c
\qquad
\tau'_G|c = (\tau_G \cdot \tau_R \cdot \tau_G)|_c
\]
In particular a green edge (corresponding to transposition of $\tau_G$) or
half-edge (corresponding to a fixed point of $\tau_G$) of $S^*(\tau)$ becomes a
red edge or respectively half-egde of $S^*(\tau')$. And we map a red edge
$(i j)$ or half-edge $(i)$ of $S^*(\tau)$ to the green edge
$(\tau_G(i) \tau_G(j)) = (i j)^{\tau_G}$ or respectively half-edge
$(\tau_G(i)) = (i)^{\tau_G}$.

We define similarly $(i,j)$-shears for any pair of colours in $\{R,G,B\}$ and we
naturally extend the definition of $\Shear_{c,i,j}$ when $c$ is a union of
$\{i,j\}$-components. Note that $\Shear_{c,R,G}$ and $\Shear_{c,G,R}$ are
different operations (they differ on the image of $\tau_B$) and we get 6
possible shears. We call \emph{horizontal cylinders} the \GB-components and
\emph{horizontal shears} the $(B,G)$ and $(G,B)$-shears. Similarly, we call
\emph{vertical cylinders} the \RG-components and \emph{vertical shears} the
$(R,G)$ and $(G,R)$-shears. Note that one could similarly define \emph{diagonal
  shears} as the $(R,B)$ and $(B,R)$-shears, however we will avoid them for the
following reasons. The geometric interpretation of a diagonal cylinder on the
quadrangulation corresponding to a square-tiled surface is unclear and there is
no apparent reason to favour one diagonal over the other. Moreover, these shears
can be realised by a sequence of horizontal and vertical shears whose length is
linear in the size of the diagonal cylinder.

We now prove that shears preserve the stratum and stabilise the Abelian components.
\begin{lemma}\label{lem:profileIsInvariant}
Let $\tau$ be a square-tiled surface, $i$, $j$ two colours in $\{R,G,B\}$ and
$c$ a $\{i,j\}$-component. Let $\tau' = \Shear_{c,i,j}(\tau)$ be the square-tiled
surface obtained after a cylinder shear. Then
\begin{enumerate}
\item $\tau_R \cdot \tau_G \cdot \tau_B = \tau'_R \cdot \tau'_G \cdot \tau'_R$,
\item $\tau$ and $\tau'$ have the same profile,
\item $\tau$ is Abelian if and only if $\tau'$ is.
\end{enumerate}
\end{lemma}

\begin{proof}
We only consider the case of $i=R$ and $j=G$ the other cases being similar.

Let $\tau = (\tau_R, \tau_G, \tau_B)$ be a square-tiled surface and $\tau' = (\tau'_R, \tau'_G, \tau'_B) = \Shear_{c,R,G}(\tau)$.
Then
\[
\tau'_R \cdot \tau'_G \cdot \tau'_B
=
(\tau_R \cdot c_R \cdot c_G) \cdot (\tau_G \cdot c_R \cdot c_G) \cdot \tau_B
=
\tau_R \cdot (c_R \cdot c_G \cdot \tau_G \cdot c_R \cdot c_G) \cdot \tau_B
\]
Now we note that $c_R \cdot c_G \cdot \tau_G = \tau_G \cdot c_G \cdot
c_R$. Indeed, outside of $c$ both sides are equal to $\tau_G$ and in $c$ both
sides are equal to $c_R$. In particular, the rightmost term in the chain of
equalities simplifies to $\tau_R \cdot \tau_G \cdot \tau_B$ and concludes the
proof of the first item.

Let $(\mu, k)$ and $(\mu', k')$ be the profiles of respectively $\tau$ and
$\tau'$.  Recall that $(\mu, k)$ is the pair made of the integer partition $\mu$
associated to the conjugacy class of the product $\tau_R \tau_G \tau_B$ and the
integer $k$ which is the sum of the number of fixed points in $\tau_R$, $\tau_G$
and $\tau_B$. By the previous item we have that $\mu = \mu'$. Next, the number
of fixed points of $\tau_B = \tau'_B$ are identical. Also, since $\tau_R$ and
$\tau_G$ are unchanged outside of $c$, the number of fixed points up to the ones
in $c$ coincide. We show that the sum of the number of green and red fixed
points in $c$ are identical. Indeed, we have $\tau'_R|_c = \tau_G|_c$ and
$\tau'_G|_c = (\tau_G \cdot \tau_R \cdot \tau_G)|_c$ and hence the number of red
and green fixed points are exchanged in $c$. This shows that the total number of
fixed points remains unchanged.

Finally, if $\tau$ admits an invariant partition $A_+ \sqcup A_-$ then the same partition
is invariant by~$\tau'$.
\end{proof}

Next, we immediately see from the definition that $\Shear_{c,R,G}$ and $\Shear_{c,G,R}$
are inverse of each other.
\begin{lemma}
Let $\tau$ be a tricoloured cubic graph, $i$, $j$ two colours in $\{R,G,B\}$ and
$c$ a $\{i,j\}$-component. Then $c$ is also a $\{i,j\}$-component in
$\Shear_{c,i,j}(\tau)$ and $\Shear_{c,j,i} (\Shear_{c,i,j}(\tau)) = \tau$.
\end{lemma}

Finally, let us conclude with an analogy between cylinder shears in square-tiled
surfaces and a reconfiguration problem on graph colourings.  An
\emph{edge-colouring} of a graph is an assignment of colours to the edges, such
that incident edges receive different colours. A Kempe change consists in
swapping two colours in a maximal bichromatic component of the colouring. So a
cylinder shear comes down to performing a Kempe change on the 3-edge-colouring
of $S^*(\tau)$, before sliding the incident edges to preserve the cyclic
ordering of the colours around the vertices of $S^*(\tau)$.

\subsection{Abelian and quadratic
  differentials} \label{ssec:AbelianAndQuadraticDifferentials} We now introduce
Abelian and quadratic differentials on Riemann surfaces and their associated
moduli spaces. We explain how Abelian and quadratic square-tiled surfaces are
particular cases of differentials and how a cylinder shear can be realised as a
continuous path in the moduli space. The main consequence of this fact is the
proof of~\cref{prop:connectedComponentsAndConnectedComponents}. Beyond the proof
of this proposition, the definitions in this section will be used only
in~\cref{sec:hyperellipticSquareTiledSurfaces}.

For a more detailed introduction to Abelian and quadratic differentials, we refer the
reader to the survey~\cite{Zorich2006} and the
book~\cite{AthreyaMasur2024}.

One often uses the term \emph{translation surface} to refer to a compact Riemann
surface endowed with a non-zero Abelian differential (or equivalently a
holomorphic one-form). The reason is that the most elementary way to define such
object is by considering a finite collection of Euclidean polygons whose edges
are identified by translations (see~\cite[Section~1.2]{Zorich2006}
and~\cite[Chapter~2]{AthreyaMasur2024}). The Abelian square-tiled surfaces we
consider in this article are indeed particular cases of translation
surfaces. More precisely, let $\tau = (\tau_R, \tau_G, \tau_B)$ be an Abelian
square-tiled surface on $[n]$ where $n$ is twice the number of
quadrilaterals. Let $A^+ \sqcup A^-$ be a $\tau$-invariant bipartition of
$[n]$. We consider $n/2$ copies of unit squares with bottom left labelled with
elements of $A^+$ and from now on, we identify the squares with their labels in
$A^+$. Recall that $[n]$ corresponds to the triangles and that the ``other
half'' of the bottom left triangle labelled $i \in A^+$ is labelled
$\tau_G(i) \in A^-$ so that a square is made of two triangles. Now we glue the
right side of the square labelled $i \in A^+$ to the left side of
$\tau_R(\tau_G(i))$ and its top side to the bottom side of $\tau_B(\tau_G(i))$.
Note that there is one ambiguity in the construction which is the choice of the
bipartition. If the partitions $A^+$ and $A^-$ are swapped, we obtain a
(possibly different) translation surface built from squares that are all rotated
by 180 degrees. In other words, we associate to an Abelian square-tiled surface
$\tau$ a translation surface up to rotation by 180 degrees that we denote by
$\pm S$.

Two translation surfaces $S$ and $S'$ are \emph{isomorphic} if one can obtain
one from the other by a sequence of cutting and gluing operations on the
polygons (see~\cite[Section~2.5.4]{AthreyaMasur2024}).
\begin{lemma}\label{lem:isomorphismAbelian}
Let $\tau$ and $\tau'$ be two Abelian square-tiled surfaces on $[n]$.
Then $\tau$ and $\tau'$ are isomorphic (for the definition from
\cref{sssec:cubicEdgeTricoloredCubicGraphs}) if and only if
the associated translation surfaces $\pm S$ and $\pm S'$ are isomorphic.
\end{lemma}

In order to prove \cref{lem:isomorphismAbelian}, we need to introduce two
additional definitions. A translation surface $S$ is naturally a metric space
where the distance is induced from the Euclidean distance on each polygon
forming $S$. Endowed with this metric, each point on a translation surface has a
neighbourhood which is isometric to the neighbourhood of the apex of a flat cone
of angle $(1 + k) 2 \pi$ where $k$ is a non-negative integer. When $k=0$ such a
point is called \emph{regular}. For $k > 0$, such a point is called a
\emph{conical singularity of angle $(1+k) 2\pi$} and $k$ is called the
\emph{excess}. Any point inside a polygon or inside an edge is regular. However,
the vertices might be regular or a conical singularity. For example, the corners
of the squares in \cref{sfig:AbSquareTiled} correspond to two points on the
associated translation surface. Both of these points have angle $4\pi$, in other
words they have excess $1$.

To $S$ we associate the integer partition $\kappa = [1^{\kappa_1},
2^{\kappa_2}, \ldots]$ where $\kappa_i$ is the number of conical singularities
of excess $i$ on $S$. The genus of $S$ is related to $\kappa$ by the following
Euler's formula $\sum \mu_i \cdot i = 2g - 2$.
For a square-tiled surface $S(\tau)$ one can obtain the
integer partition $\kappa$ from its profile $\mu$ as $\kappa = [1^{\mu_4},
2^{\mu_6}, \ldots]$.

Beyond the metric, one can make sense of \emph{straight line segments} in $S$
and associate to a straight line segment a \emph{direction}. One uses the
standard Euclidean notion in each polygon of $S$ and such definition can be made
global since the gluings are done by translation.  Namely, segments can be
defined on each polygon $S$ is made of and pieces of segments ending in middle
of edges of the polygons can be glued together to form longer segments. In a
square-tiled surfaces, the red, blue and green edges are examples of horizontal
segments, vertical segments and segments of slope -1. A \emph{saddle connection}
in $S$ is a segment both of whose endpoints are conical singularities (possibly
the same).

We are now ready to prove \cref{lem:isomorphismAbelian}.

\begin{proof}
  Let us first remark that the number of triangles $n$ of $\tau$ and $\tau'$ is
  twice the number of squares, which is the areas of $S$ and $S'$.

  It is clear that if $\tau$ and $\tau'$ are isomorphic then $\pm S$ and
  $\pm S'$ are isomorphic. Indeed, the labelling of the squares is irrelevant in
  the geometric structure.

  We now focus on the other implication. Namely, assume that $S$ and $S'$ are
  isomorphic as translation surfaces. Let us first remark that the isomorphism
  of translation surfaces preserves area and singularities. Hence the number of
  squares in $\tau$ and $\tau'$ are identical as well as the vector of excesses
  $\kappa$ and $\kappa'$.

  In order to prove that $\tau$ and $\tau'$ are isomorphic we show that the
  triangulation of $S$ induced by $\tau$ can be reconstructed from the geometric
  structure of $S$ only. It will result that if $S$ and $S'$ are isomorphic,
  then the underlying triangulations are isomorphic and so are $\tau$ and
  $\tau'$.

  Let us first assume that $S$ has at least one conical singularity (ie it is
  not a torus in $\cH(\emptyset)$). We consider in $S$ the horizontal and
  vertical saddle connections issued from these conical singularities. These
  saddle connections are respectively unions of red and blue edges of the
  underlying triangulation. The lengths of each of these saddle connections is
  integral and we mark the regular points at each unit length along them. These
  marked points are vertices of $\tau$. We then consider horizontal and vertical
  saddle connections issued from these marked points and iterate the process
  until there is no new segment discovered. Each time a new regular point is
  marked it is a vertex of $\tau$, and each horizontal and vertical saddle
  connection built is a subset of the red and blue edges of the triangulation
  respectively. By connectedness of the triangulation, all red and blue edges
  are recovered in this way. The red and blue edges form a square tiling and one
  obtains the green edges by picking the north-west to south-east diagonals
  (which is the same diagonal on $S$ and $-S$). This concludes the case when $S$
  is not a flat torus.

  We now consider the case of $S$ being a flat torus. In this case $S$ is
  isomorphic to a quotient $\bR^2 / \Lambda$ where $\Lambda$ is a lattice and
  hence admits a group structure.  For each element $x \in \bR^2 / \Lambda$ the
  translation $y \mapsto x + y$ is an isomorphism of $S$. In particular, given
  any two points $x$ and $y$ on $S$ there exists a unique isomorphism of $S$
  that maps $x$ to $y$ (the translation by $y - x$).  Let us choose an arbitrary
  point in $S$ that we declare as a marked point. By the previous discussion,
  the choice of such point is geometrically irrelevant. From there, we can
  perform the construction of the previous paragraph starting at this regular
  marked point.  We obtain a triangulation which is isomorphic to the one
  obtained from $\tau$.
\end{proof}

The set of isomorphism classes of the translation surfaces with fixed vector
$\kappa$ forms an algebraic variety, in particular a topological space which is
denoted $\cH(\kappa)$ (see~\cite[Section~3.1]{Zorich2006}
and~\cite[Chapter~3]{AthreyaMasur2024}). It follows that $\cH(\kappa)$ has
finitely many connected components, which are also path-connected. These
connected components have been classified in~\cite{KontsevichZorich2003}.  The
topology of $\cH(\kappa)$ is easy to describe: two surfaces are nearby if they
are obtained by the same gluing patterns of nearby collections of polygons.

\begin{lemma}\label{lem:cylinderShearIsAPath-Abelian}
  Let $S$ in $\cH(\kappa)$ be the translation surface associated to an Abelian
  square-tiled surface $\tau$. Let $S'$ be the translation surface associated to
  the square-tiled surface $\tau' = \Shear_{c,i,j}(\tau)$ obtained after a cylinder
  shear, where $c$ is some $\{i,j\}$-component of $\tau$. Then $S'$ belongs to
  $\cH(\kappa)$ and there is a continuous path from $S$ to $S'$ in
  $\cH(\kappa)$.
\end{lemma}

\begin{proof}
  By \cref{lem:profileIsInvariant}, the profile of $\tau'$ is the same as the
  one of $\tau$ and $\tau'$ is Abelian. Hence the translation surface $S'$
  belongs to the same stratum $\cH(\kappa)$ as $S'$. Though, we will also obtain
  this fact from the explicit construction of a continuous path from $S$ to
  $S'$.

  Now, our goal is to build a continuous path of translation surfaces
  $\{S_t\}_{t \in [0,1]}$ inside $\cH(\kappa)$ such that $S_0 = S$ and
  $S_1 = S'$. The $\{i,j\}$-component $c$ in $\tau$ corresponds to a horizontal
  cylinder in the translation surface $S$. More precisely, the closure of the
  subset of squares corresponding to this $\{i,j\}$-component forms a cylinder
  of height one. We define the surface $S_t$ to be the translation surface built
  from an identical number of quadrilaterals where the quadrilaterals outside of
  $c$ remains squares but the one in $c$ are the ``slanted'' quadrilaterals with
  vertices $[(0,0),(1,0),(1+t,1),(t,0)]$. It is easy to notice that $S_1$ and
  $S'$ are indeed isomorphic.
\end{proof}

We now turn to quadratic differentials on Riemann surfaces which are sometimes
referred to as \emph{half-translation surfaces}. One can build a
half-translation surface by taking finitely many Euclidean polygons and gluing
their edges using translations and 180-degree rotations
(see~\cite[Section~8.1]{Zorich2006}
and~\cite[Section~2.7]{AthreyaMasur2024}). Similarly to the Abelian case, we
associate to a half-translation surface an integer partition with parts in
$\{-1, 1, 2, 3, \ldots\}$ denoted
$\kappa = [(-1)^{\kappa_{-1}}, 1^{\kappa_1}, \ldots]$ such that there are
$\kappa_i$ singularities of angle $(2 + i) \pi$. Euler's formula in the
quadratic case reads $\sum i \cdot \kappa_i = 4g-4$. Similarly to the Abelian
case, there exist strata $\cQ(\kappa)$ that consist of quadratic strata that are
not squares of Abelian differentials.

Given a quadratic square-tiled surface in $\ST_{quad}(\mu, k)$ on $[n]$, one
obtains a half-translation surface built from $n$ right-angled triangles (which
should be thought of as ``half squares''). The vector $\kappa$ associated to
this half-translation surface has integer partition
$[(-1)^{k + \mu_1}, 1^{\mu_3}, 2^{\mu_4}, \ldots]$. Note that both half-edges
(corresponding to $k$) and vertices of degree $1$ (corresponding to $\mu_1$)
give rise to singularities of angle $\pi$.

Similarly to \cref{lem:cylinderShearIsAPath-Abelian}, we have the following.
\begin{lemma}\label{lem:cylinderShearIsAPath-quadratic}
  Let $S$ in $\cQ(\kappa)$ be the half-translation surface associated to a
  quadratic square-tiled surface $\tau$. Let $S'$ be the half-translation
  surface associated to the square-tiled surface $\tau' = \Shear_{c,i,j}(\tau)$
  obtained after a cylinder shear, where $c$ is some $\{i,j\}$-component of
  $\tau$. Then $S'$ belongs to $\cQ(\kappa)$ and there is a continuous path from
  $S$ to $S'$ in $\cQ(\kappa)$.
\end{lemma}

\begin{proof}
  By \cref{lem:profileIsInvariant}, $\tau'$ is quadratic and its profile is the
  same as the one of $\tau$. Hence, the half-translation surface $S'$ belongs to
  the same stratum $\cQ(\kappa)$ as $S'$.

  It is straightforward to extend the proof of
  \cref{lem:cylinderShearIsAPath-Abelian} to the quadratic case when $c$ is a
  cycle. In the case where $c$ is a path, the union of triangles corresponding
  to $c$ in $S$ is still a horizontal cylinder, though of height $1/2$. The same
  shearing construction works.
\end{proof}

We can now prove \cref{prop:connectedComponentsAndConnectedComponents}, which
states that the connected components induced by the cylinder shears refine the
connected components of the strata of the moduli space of quadratic
differentials.
\begin{proof}[Proof of \cref{prop:connectedComponentsAndConnectedComponents}]
  By \cref{lem:cylinderShearIsAPath-Abelian} and
  \cref{lem:cylinderShearIsAPath-quadratic} cylinder shears can be realised as a
  continuous path in respectively $\cH(\kappa)$ and $\cQ(\kappa)$.  In
  particular, if $\tau$ belongs to a given connected component of $\cH(\kappa)$
  or $\cQ(\kappa)$ its image under a cylinder shear belongs to the same
  connected component.
\end{proof}

\subsection{Cylinder decomposition and weighted stable
  graphs} \label{sssec:weightedStableGraph} We introduce now the weighted stable
graph of a square-tiled surface that encodes the geometry of its
\GB-components. It follows closely the notions
in~\cite[Section~4.7]{delecroix2020Enumeration}
and~\cite[Section~2.2]{DelecroixGoujardZografZorich2021}. The weighted stable
graph will be the main tool in \cref{sec:connectingtoPathLike} but will not be
used elsewhere.

Let $\tau$ be a square-tiled surface. Recall from
\cref{ssec:AbelianAndQuadraticDifferentials} that when viewed as the gluing of
right-angled triangles, $S(\tau)$ carries the geometry of an Abelian or
quadratic differential. In such a surface, it makes sense to consider horizontal
segments: these are continuous paths that are horizontal in every triangle. The
\emph{critical graph} of $S(\tau)$ is the graph embedded in the surface
$S(\tau)$ whose vertices are the singularities of $S(\tau)$ (i.e.\ the union of
vertices of $\tau$ whose degree is different from 6 and the mid-points of
half-edges) and its edges are the union of all horizontal segments ending at
these singularities. A connected component of the critical graph is a saddle
connection that either consists entirely of red edges or half-edges of $S(\tau)$
or connects two conical singularities of angle $\pi$. These saddle connections
are in bijection with the horizontal cylinders in $S^*(\tau)$.

\begin{figure}[!ht]
\centering
\begin{subfigure}{.55\linewidth}
\includegraphics{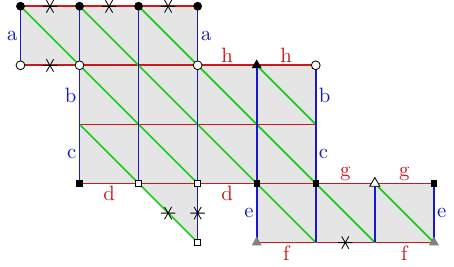}
\caption{A spherical square-tiled surface in $\ST_{quad}([1^3, 2^6, 3^2, 4^2], 7)$.}
\label{sfig:stableTreeExample1}
\end{subfigure}
\hfill
\begin{subfigure}{.4\linewidth}
\includegraphics{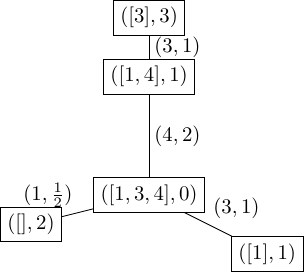}
\caption{The weighted stable graph of \cref{sfig:stableTreeExample1}.}
\label{sfig:stableTreeExample2}
\end{subfigure}
\caption{A square-tiled surface $S(\tau)$ and its weighted stable graphs $\Gamma(\tau)$.}
\end{figure}

The \emph{weighted stable graph} associated to $\tau$ is the multigraph
$\Gamma(\tau)$ together with vertex decorations $(\mu^{(v)}, k^{(v)})$ for each
vertex $v \in V(\Gamma(\tau))$ and edge decorations $(w_e, h_e)$ for each edge
$e \in E(\Gamma(\tau))$ built as follows.

\begin{itemize}
\item $V(\tau)$ is the set of connected components of the critical graph of
  $S(\tau)$; the vertex decoration $(\mu^{(v)}, k^{(v)})$ associated to a
  component records the singularity pattern of that component ignoring regular
  vertices (corresponding to $\mu_2$),
\item $E(\tau)$ is the set of cylinders, i.e. the connected components of
  $S(\tau)$ minus the critical graph. The two ends of an edge are the two
  connected components of the critical graph to which the boundaries of
  cylinder are glued to; the edge decoration
  $(w_e, h_e) \in \bZ_{> 0} \times \frac{1}{2} \bZ_{> 0}$ records the width and
  height of the cylinder. Note that this allows loops.
\end{itemize}
By construction
\begin{itemize}
\item the union of the vertex decorations $(\mu^{(v)}, k^{(v)})$ is the
  \emph{reduced profile} of $\tau$, that is the profile $\mu$ of $\tau$ where
  the singularities corresponding to $\mu_2$ are omitted.
\item the number of triangles $n = \sum_{i \ge 0} \mu_i$ of $S(\tau)$ satisfies
  $\sum_{e \in E(\Gamma)} w_e \cdot h_e = \frac{n}{2}$.
\end{itemize}

Let us terminate this subsection with three remarks. First, we note that
one can read the genus of $S(\tau)$ from $\Gamma(\tau)$. Namely, to each vertex
$v \in V(\Gamma)$ one can associate the genus of the corresponding singular
layer $g^{(v)}$ which is obtained from the vertex decoration and the degree of
$v$ in $\Gamma(\tau)$ as
\[
4 g^{(v)} - 4 = \sum_i \mu^{(v)}_i \cdot (i - 2) - k^{(v)} - 2 \deg(v).
\]
Then the genus of the surface $S(\tau)$ is
$\sum_{v \in \Gamma(\tau)} g^{(v)} + |E(\Gamma(\tau))| - |V(\Gamma(\tau))| +
1$. In particular, for a spherical square-tiled surfaces, $\Gamma(\tau)$ is
always a tree.

Secondly, if $S(\tau)$ and $S(\tau')$ are two square-tiled surfaces that differ
by some horizontal shears then their weighted stable graphs $\Gamma(\tau)$ and
$\Gamma(\tau')$ are isomorphic.

Finally, if $S(\tau)$ is a square-tiled surface and $c_1$ and $c_2$ are two
\GB-components that correspond to the same edge in $\Gamma(\tau)$ then the
associated cylinder shears produce isomorphic square-tiled surfaces (see
\cref{fig:neighborComponents}).
\begin{figure}[h]
\begin{subfigure}[t]{.48\linewidth}
\includegraphics{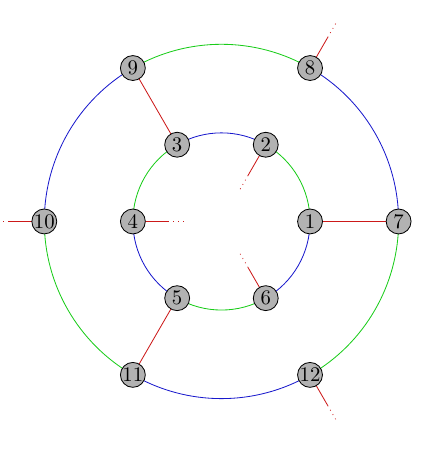}
\caption{Two \GB-cycles of length 6 (separated by three hexagons) whose shearing
  produces isomorphic square-tiled surfaces.}
\end{subfigure}
\hfill
\begin{subfigure}[t]{.48\linewidth}
\includegraphics{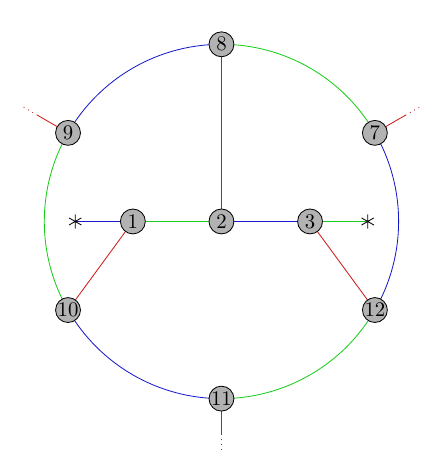}
\caption{A \GB-path of length 3 and a \GB-cycle of length 6 whose
  shearing produces isomorphic square-tiled surfaces.}
\end{subfigure}
\caption{Equivalence of cylinder shears.}
\label{fig:neighborComponents}
\end{figure}

\section{Connecting to path-like configurations}
\label{sec:connectingtoPathLike}
In this section, we focus on the square-tiled surfaces of genus 0, also called
\emph{spherical square-tiled surfaces}. These square-tiled surfaces are exactly
those whose profile is \emph{spherical}, i.e. satisfies
$\sum_i (i-2)\mu_i = k-4$. We will assume that it is the case in this whole
section.

\subsection{Path-like configurations}\label{ssec:pathLike}

A square-tiled surface is called a \emph{path-like configuration} if it has a
single horizontal cylinder (that is, a \GB-component) which furthermore is a
path.  The path-like square-tiled surfaces correspond to some specific
Jenkins-Strebel differentials appearing in~\cite{Zorich2008} and to the
one-cylinder square-tiled surfaces studied
in~\cite{DelecroixGoujardZografZorich2020}.

Note that if $(\mu, k)$ is the profile of a path-like configuration then
necessarily $k \ge 2$ because of the two ends of the path. One important result
of this subsection is that this condition is the only restriction on the profile
for the existence of path-like square-tiled surfaces.
\begin{theorem}\label{thm:pathLikeExistence}
 Let $(\mu, k)$ be a spherical profile, with $k \ge 2$. Then there exists a
 path-like configuration with profile $(\mu, k)$.
\end{theorem}

In order to prove \cref{thm:pathLikeExistence} we construct a surjective map
between path-like configurations and certain decorated plane trees, that we
define below. This map will be used in \cref{ssec:reconfigurationPathLike} to
reconfigure path-like square-tiled surfaces.

A \emph{decorated plane tree} is a pair $(T, L)$ where $T$ is a tree embedded in
the sphere and $L$ is a subset of the leaves of $T$. We say that $(T, L)$ has
\emph{profile} $(\mu, \ell)$ if $L$ has size $\ell$ and the vertices of $T$ not
in $L$ have degrees given by the integer partition $\mu$. We will call
\emph{$L$-leaves} the leaves in $L$ and \emph{$\overline{L}$-leaves} the other
leaves. The \emph{perimeter} of a decorated plane tree is its number of edges
incident to an $L$-leaf plus twice the number of other edges.

Let $\tau$ be a path-like square-tiled surface. We explain how to associate to
$\tau$ its decorated plane tree. The vertex set of $T$ contains the vertices of
the surface $S(\tau)$, together with the set $L$ containing one vertex for each
half-edge of $S(\tau)$. To obtain $T$, it suffices to keep only the red edges of
$S(\tau)$, and to replace the red half-edges by an edge connecting its endpoint
to a $L$-leaf. Note that $(T,L)$ can also be defined by taking the dual of the
red edges in the cubic graph $S^*(\tau)$. The main result of this subsection and
argument in the proof of \cref{thm:pathLikeExistence} is the following
relationship between path-like configurations and decorated plane trees.
\begin{theorem}\label{thm:dualTree}
  Let $(\mu, k)$ be a spherical profile with $k \ge 2$. Then the map which takes
  a path-like configuration with profile $(\mu, k)$ to its decorated plane tree
  is a surjection onto decorated plane trees with profile $(\mu,
  k-2)$. Furthermore, for each decorated plane tree $(T,L)$ with profile
  $(\mu,k-2)$ the corresponding preimage of path-like configurations is made of
  a single orbit under the horizontal shears.
\end{theorem}

\begin{proof}
  Let $\tau$ be a path-like configuration with profile $(\mu, k)$. We first
  prove that its image is indeed a decorated plane tree. The only thing to prove
  is that $T$ is indeed a tree. As $\tau$ is spherical and $S(\tau)$ has only
  one \GB-component that is a path, hence not separating, the complement of $T$
  in the sphere is indeed a disk.

  To obtain the surjectivity, we now explain how to build the tricoloured cubic
  graph dual to a path-like configuration that maps to a given pair $(T, L)$
  with profile $(\mu, k-2)$. We first consider the arch configuration dual to
  $(T, L)$ on a disk as in \cref{fig:blueTreeAndDual}. Here by \emph{arch
    configuration} we mean a choice of a non-intersecting matching (or
  involution) on cyclically ordered points.  An arch configuration has a
  geometric realization as a union of arcs of circles (for each pair of the
  matching) and segments with one end on the boundary (for points not in the
  matching). Then in another disk, we consider a \GB-path on $n$ vertices and
  add red segments going to the boundary as the red picture in
  \cref{fig:blueTreeAndDual} where $n$ is the perimeter of $(T,L)$. Finally we
  obtain a spherical square-tiled surface by gluing together the two disks as in
  \cref{fig:matchingDisks}. Note that there are $n$ possible choices of gluings
  and each of them gives a square-tiled surface with the same dual decorated
  plane tree $(T,L)$.

  We now prove that the set of square-tiled surfaces with decorated plane tree
  $(T, L)$ forms a single orbit under horizontal shears. Let $S^*(\tau)$ be a
  tricoloured cubic graph dual to a given $\tau$ with decorated plane tree
  $(T,L)$. We will mark $(T,L)$ to track the position of the \GB-path of
  $S^*(\tau)$ compared to the red edges. We associate a \emph{marking site} to
  each endpoint of the red edges of $S^*(\tau)$ (each half-edge receives only
  one marking site instead of two). In $(T,L)$, this corresponds to placing a
  marking site on each edge of $T$ adjacent to an $L$-leaf and one on each side
  of the other edges. We now need to distinguish two cases depending on the
  parity of the perimeter of $(T,L)$, that is the number of marking sites in
  $(T,L)$. Note that $(T,L)$ has an even perimeter if and only if the half-edges
  at the endpoint of the \GB-path of $S^*(\tau)$ (and all other preimages of
  $(T,L)$) have identical colours.

  Assume that they have distinct colours, this corresponds to the case where
  $(T,L)$ has odd perimeter. Mark the marking site adjacent to the unique green
  half-edge of the \GB-path of $S^*(\tau)$. The marking sites of $(T,L)$ appear
  in a cyclic ordering on the boundary of the complement of $T$. Notice that
  performing two $(G,B)$-shears along the unique \GB-path of $S^*(\tau)$ moves
  the mark to the next marking site in the anti-clockwise order. This proves
  that the preimages of $(T,L)$ form a single orbit under the horizontal shears
  and that the number of such preimages is the perimeter of $(T,L)$.

  Assume now that the perimeter is even, this corresponds to the case where the
  half-edges at the endpoints of the \GB-path of $S^*(\tau)$ have identical
  colours. Up to performing one $(G,B)$-shear (which has no effect on $(T,L)$),
  we can assume that both are green. We place two marks on $(T,L)$, one next to
  each half-edge of the \GB-path of $S^*(\tau)$. These marks are
  \emph{antipodal}: the number of marking sites between them on each side of the
  boundary of the complement of $T$ is half the perimeter of $(T,L)$. Like
  before, performing two $(G,B)$-shears along the unique \GB-path of $S^*(\tau)$
  moves both marks to the next marking site in the anti-clockwise order. Since
  all the markings corresponding to a preimage with two green half-edges are
  antipodal, this proves that these the preimages are in a single orbit under
  the horizontal shears. Moreover, the number of such preimages is the half
  perimeter of $(T,L)$. The same is true for preimages with two blue half-edges,
  and one alternates between one case and the other with each $(G,B)$-shear.
\end{proof}
\begin{figure}[!ht]
\begin{center}
\includegraphics{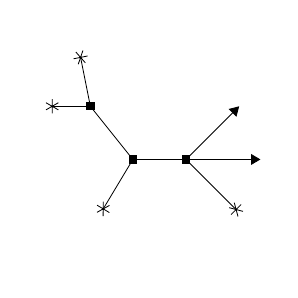}
\includegraphics{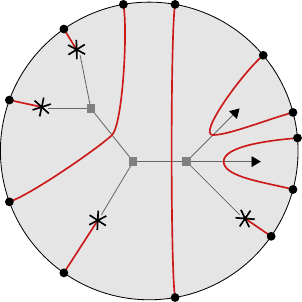}
\hspace{.4cm}
\includegraphics{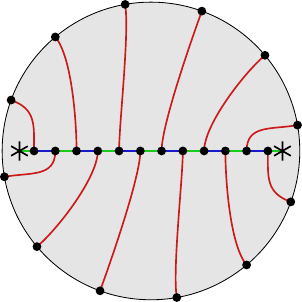}
\end{center}
\caption{A decorated plane tree $(T, L)$ with profile $([4,3^2,1^2],4)$, its
  dual arch configuration and a \GB path in a disk with same diameter. When
  drawing a decorated plane tree, we will represent the $L$-leaves by a star,
  and the other leaves by triangular nodes.}
\label{fig:blueTreeAndDual}
\end{figure}

\begin{figure}[!ht]
\begin{center}
\includegraphics{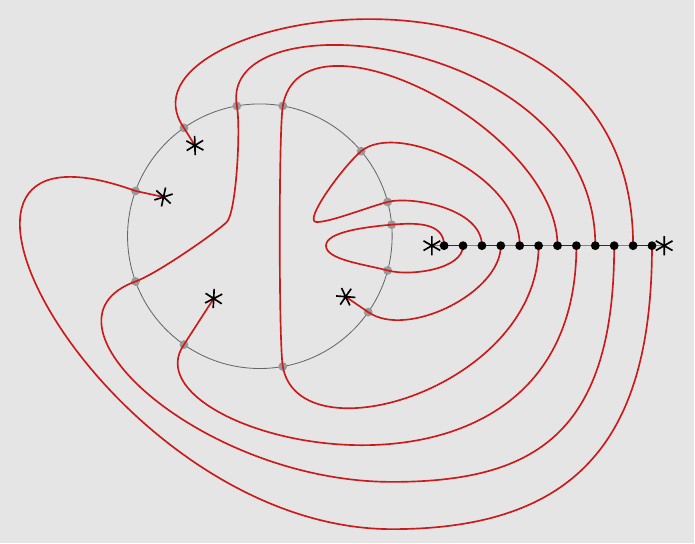}
\end{center}
\caption{Matching the arch configuration and the \GB-path of
  \cref{fig:blueTreeAndDual} to reconstruct a tricolored cubic graph as in the
  proof of \cref{thm:dualTree}.}
\label{fig:matchingDisks}
\end{figure}

\cref{thm:pathLikeExistence} directly follows from
\cref{thm:dualTree} and the following lemma.
\begin{lemma}
\label{lem:existenceDecoratedPlaneTree}
Let $(\mu, k)$ be a spherical profile with $k \ge 2$.
Then there exists a decorated plane tree with profile $(\mu, k-2)$.
\end{lemma}

\begin{proof}
  It is a standard result that a tree on $n$ vertices with a given degree
  sequence $d_1, d_2, \ldots, d_n$ exists if and only if
  $d_1 + d_2 + \ldots + d_n = 2(n-1)$. Now for a tree with profile $(\mu, k-2)$
  the number of vertices is $n=\sum \mu_i + k - 2$ and the degree sequence is
  $d = [1^{\mu_1+k-2},2^{\mu_2},\ldots]$. The planarity condition can be
  rewritten as
  $$
  \sum_{i=1}^n d_i = k - 2 + \sum i \mu_i = 2 \sum \mu_i + 2k - 6 = 2(n-1).
  $$
  Thus there exists a tree with degree sequence $d$. In order to build a
  decorated plane tree, one just needs to pick any planar embedding and choose a
  subset of $k-2$ leaves.
\end{proof}

Note that the proof of \cref{thm:dualTree} also provides a way of counting the
number of path-like configurations in a stratum. For each decorated plane tree
$(T,L)$ in the stratum, the number of preimages, up to automorphism, is equal to
the perimeter of $(T,L)$ divided by its number of automorphisms.

Finally, the following proposition is an immediate consequence of the existence
of path-like configurations and implies the lower bound in
\cref{thm:connectednessSphere}.
\begin{proposition}\label{prop:squareTiledLowerBound}
  Let $(\mu, k)$ be a spherical profile, with $k \ge 2$. Then there are
  square-tiled surfaces with profile $(\mu, k)$ that are separated by $\Omega(k)$
  cylinder shears.
\end{proposition}
\begin{proof}
  A shear can only change by at most two the number of half-edges coloured red
  in $S^*(\tau)$. By \cref{thm:pathLikeExistence}, there is a path-like square
  tiled-surface $\tau$ of profile $(\mu,k)$ and by symmetry, there is also a
  square tiled-surface $\tau'$ of profile $(\mu,k)$ that has a single vertical
  cylinder which furthermore is a path. Since $\tau$ and $\tau'$ have
  respectively $k-2$ and at most two red half-edges, this concludes the proof.
\end{proof}

\subsection{Connecting to a path-like configuration}
\label{sssec:connectingToPathLike}
We now state and prove a general result about the reduction to path-like
configurations of spherical square-tiled surface (see \cref{lem:movesToPathLike}
below). Let us emphasise that it requires an assumption on the profile that is
not satisfied in general by spherical square-tiled surfaces, namely
$\mu_1 \le 1$. It is a partial step towards \cref{thm:connectednessSphere}.

\begin{lemma} \label{lem:movesToPathLike} Let $\tau$ be a spherical square-tiled
  surface with profile $(\mu,k)$, such that $\mu_1 \le 1$. Then $\tau$ can be
  connected to a path-like square-tiled surface $\tau'$ with a sequence of
  $O(k)$ vertical cylinder shears and powers of horizontal cylinder shears.
\end{lemma}

The main technique to transform a spherical square-tiled surface into a
path-like configuration is the fusion that we introduce now.  Let $\tau$ be a
square-tiled surface. A \emph{fusion path} in $\tau$ is a vertical (or \RG)
component of $S^*(\tau)$ that is a path and such that
\begin{itemize}
\item its intersection with each horizontal (or \GB) cycle is either empty or a
  single green edge,
\item its intersection with each horizontal (or \GB) path is either empty or a
  single green half-edge.
\end{itemize}
Note that by definition a fusion path intersects at most two horizontal paths.

\begin{figure}[!ht]
  \begin{subfigure}[t]{.48\textwidth}
    \centering
    \includegraphics{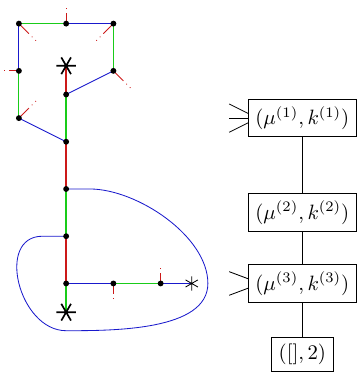}
    \caption{A fusion path that intersects three \GB-components (two
  cycles and a path).}
  \end{subfigure}
\hfill
  \begin{subfigure}[t]{.48\textwidth}
    \centering
    \includegraphics{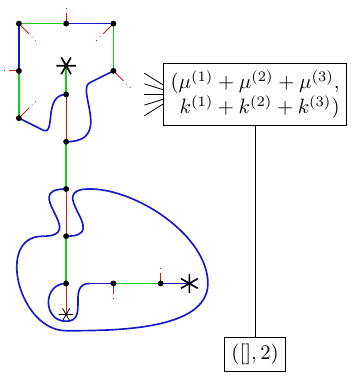}%
    \caption{The tricoloured cubic graph $S^*(\Shear_{c,R,G}(\tau))$.}
    \label{fig:fusionPath}
\end{subfigure}
\caption{A vertical shear on a fusion path results in merging the horizontal
  cylinders it crosses. To the right of each square-tiled surface $\tau$ and
  $\Shear_{c,R,G}(\tau)$ we draw the corresponding part of the stable graphs
  $\Gamma(\tau)$ and $\Gamma(\Shear_{c,R,G}(\tau))$.}
\end{figure}

\begin{lemma} \label{lem:fusionPath} Let $\tau$ be a square-tiled surface with
  $n$ triangles and $c$ be a fusion path in $\tau$. Let $U$ be the subset of
  $[n]$ that consists of the elements which belong to the vertical cylinders of
  $\tau$ which intersect $c$. Then in the square-tiled surfaces
  $\Shear_{c,R,G}(\tau)$ and $\Shear_{c,G,R}(\tau)$ the vertices of $U$ form a
  single \GB-component (see \cref{fig:fusionPath}).
\end{lemma}
Let us note that \cref{lem:fusionPath} does not require planarity.
\begin{proof}
  Let $\tau$ be a square-tiled surface and $c$ a fusion path.  We prove the case
  $\Shear_{c,R,G}(\tau)$, the case of $\Shear_{c,G,R}(\tau)$ being similar.

  By relabelling $\tau$, we assume that the vertices of $c$ are labelled $1$,
  $2$, \ldots, $m$ in such way that $\tau_R(i)$ and $\tau_G(i)$ belong to
  $\{i-1,i,i+1\}$. In particular, the two half-edges at the endpoints of $c$ are
  adjacent to $1$ and $m$.

  Each horizontal (or \GB) cycle that intersects $c$ does it along an edge with
  vertices $i$ and $\tau_G(i) = i+1$. We let $H$ denote the ordered list of
  vertices on the \GB-cycle starting at $\tau_G(i)$ and ending at $i$. We let
  $H'$ denote the same list where the first and last elements $\tau_G(i)$ and
  $i$ have been exchanged.

  Let $s$ be the number of \GB-cycles intersecting the fusion path $c$ (we have
  $m - 2 \le 2 s \le m$). We label these cycles from $1$ to $s$ according to the
  ordering of the vertices of $c$ induced by our choice of vertex labelling of
  $c$ and denote $H_1$, $H_2$, \ldots, $H_s$ their associated ordered lists.

  If the vertex $1$ in $c$ is adjacent to a green half-edge, then we denote
  $P_1$ the list of vertices of the \GB-path containing $1$ so that it ends with
  $1$.  If not, we let $P_1$ be the empty list.  If the vertex $m$ in $c$ is
  adjacent to a green half-edge, then we denote $P_m$ the list of vertices of
  the \GB-path containing $m$ so that it starts with $m$.  If not, we let $P_m$
  be the empty list.

  The set $U$ is the disjoint union of the elements in $P_1$, $P_m$, $H_1$,
  \ldots, $H_m$.  After a $(R,G)$-shear in $c$, the \GB-path becomes the
  concatenation (in that order) of
  $P_1 \cdot H'_1 \cdot H'_2 \cdots H'_s \cdot P_m$.
\end{proof}

Let us recall that in \cref{sssec:weightedStableGraph}, we introduced the
weighted stable graph $\Gamma(\tau)$ associated to a square-tiled surface
$\tau$. Let us note that $\tau$ is path-like if and only if $\Gamma(\tau)$ is a
graph made of two vertices, one of them having decoration $([], 2)$ and a unique
edge between these two vertices that carries the decoration $(2n, \frac{1}{2})$.

We now reformulate \cref{lem:fusionPath} in terms of the weighted stable graph $\Gamma(\tau)$.
\begin{lemma} \label{lem:fusionPathDualGraph} Let $\tau$ be a spherical
  square-tiled surface and let $c$ be a fusion path in $\tau$. In $S^*(\tau)$,
  each green edge of $c$ belongs to a unique \GB-cycle and each green half-edge
  of $c$ to a unique \GB-path. These \GB-components form a path $c'$ in
  $\Gamma(\tau)$ that starts and ends at vertices with decorations
  $(\mu^{(start)}, k^{(start)})$ and $(\mu^{(end)}, k^{(end)})$ such that
  $k^{(start)} > 0$ and $k^{(end)} > 0$.

  Let $\tau' = \Shear_{c,R,G}(\tau)$. Then $\Gamma(\tau')$ is obtained from
  $\Gamma(\tau)$ by replacing all vertices in $c'$ with two vertices $u$ and $v$
  with decorations $(\mu^{(u)}, k^{(u)}) = ([], 2)$ and
  $(\mu^{(v)}, k^{(v)}) = \sum_{i \in c'} (\mu^{i}, k^{i}) - ([], 2)$ connected
  by a single edge and all neighbours of vertices in $c$ in $\Gamma(\tau)$ get
  plugged to $v$.
\end{lemma}

\begin{proof}
  The trace of $c$ in $\Gamma(\tau)$ is a sequence of adjacent edges that do not
  repeat. Because $\Gamma(\tau)$ is a tree, this sequence of edges is an induced
  path. By definition of a fusion path, each endpoint of $c$ is either a red
  half-edge or a green half-edge. In both cases, the associated vertex in
  $\Gamma(\tau)$ has a signature with a positive number $k$ of half-edges.

  The operation of fusion on $\Gamma(\tau)$ can be directly read from the
  description in \cref{lem:fusionPath}.
\end{proof}

Finally, we state a lemma that allows us to find fusion paths in spherical
square-tiled surfaces. It is a bit more general than what will be used to prove
\cref{lem:movesToPathLike}.
\begin{lemma} \label{lem:fusionPathExistence} Let $\tau$ be a spherical
  square-tiled surface made of $n$ triangles. Let $\Gamma(\tau)$ be the
  associated weighted stable tree. We assume that for each leaf
  $v \in V(\Gamma)$ the associated partition $(\mu^{(v)}, k^{(v)})$ satisfies
  $k^{(v)} \ge 1$. Let $i \in [n]$ such that it is either adjacent to a red
  half-edge in $\tau$ (i.e. $\tau_R(i) = i$) or such that $i$ and its red
  neighbour $\tau_R(i)$ belong to distinct horizontal (or \GB) components. Then,
  there exists a square-tiled surface $\tau'$ obtained from $\tau$ by doing
  powers of horizontal cylinder shears such that the vertical component in
  $\tau'$ containing the image of this red edge or half-edge is a fusion path
  intersecting horizontal components. Furthermore, $\tau'$ can be chosen so that
  this fusion path ends at a leaf of the weighted stable tree
  $\Gamma(\tau) = \Gamma(\tau')$. Moreover, the number of powers of horizontal
  cylinder shears performed to obtained $\tau'$ from $\tau$ is at most the
  number of vertices in the image of the resulting fusion path in
  $\Gamma(\tau')$
\end{lemma}

\begin{proof}
  Let us first consider the case of $i$ being adjacent to a red half-edge, that
  is $i = \tau_R(i)$. We denote $e_0$ this red half-edge and our aim is to
  extend it to a fusion path. If the horizontal (or \GB) component of
  $S^*(\tau)$ containing $i$ is a path, then one can perform a power of a
  horizontal cylinder shear on this path until $e_0$ becomes adjacent to a green
  half-edge. This \RG-path (made of two half-edges) is a fusion path. Let us
  assume now that $i$ belongs to a \GB-cycle of $S^*(\tau)$ that we denote
  $h_1$.  By the Jordan Curve theorem, $h_1$ separates the sphere into two
  distinct connected components.  If the connected component that does not
  contain $e_0$ is a terminal component, i.e. a leaf of $\Gamma(\tau)$, then it
  contains a red half-edge $e_1$. Indeed, by assumption we have that
  $k^{(v)} > 0$ for every leaf $v \in \Gamma(\tau)$.  In that case, one can
  apply a power of a horizontal cylinder shear on $h_1$ so that $e_0$ and $e_1$
  become adjacent to a common green edge of $h_1$.  This terminates the
  construction in that situation. Now, if the connected component is not a leaf
  of $\Gamma(\tau)$ then there exists a red edge $e_1$ in that connected
  component connecting $h_1$ to a distinct \GB-component $h_2$. By performing a
  horizontal cylinder shear on $h_1$ one can make $e_0$ and $e_1$ adjacent to a
  common green edge. We can continue this construction starting from $e_1$
  instead of $e_0$ and obtain a fusion path.

  To handle the case of $i \not= \tau_R(i)$, we cut this red edge into two
  half-edges $e_0$ and $e'_0$ (the constructed square-tiled surface remains
  planar). The previous situation allows to build two fusion paths containing
  these two half-edges in some square-tiled surface $\tau'$. Because $\tau'$ was
  obtained by performing only horizontal cylinder shears, $e_0$ and $e'_0$
  belong to the same face of $\tau'$ and we can glue them back and obtain a
  fusion path in some square-tiled surface $\tau''$. It is easy to see that
  $\tau''$ is obtained from $\tau$.

  Note that at most one power of a horizontal cylinder shear was performed for
  each horizontal component intersecting the fusion path. We claim that it is
  enough to do one such power in each horizontal cylinder of $S(\tau)$. Indeed,
  if two horizontal components of the cubic graph $S^*(\tau)$ belong to the same
  cylinder of $S(\tau)$ then performing a cylinder shear on one or the other
  results in isomorphic square-tiled surfaces. This concludes the proof.
\end{proof}

We are now ready to prove \cref{lem:movesToPathLike}.
\begin{proof}[Proof of \cref{lem:movesToPathLike}]
  Let $S(\tau)$ be a spherical square-tiled surface with profile $(\mu, k)$ with
  $\mu_1 \le 1$. By Euler characteristic consideration, for each leaf $v$ of the
  stable tree $\Gamma(\tau)$ we have $k^{(v)} + \mu^{(v)}_1 \ge 2$. The assumption
  on $\mu_1$ hence implies that $k^{(v)} \ge 1$ for each leaf $v$. In particular,
  $\tau$ satisfies the assumptions of \cref{lem:fusionPathExistence}.

  Assume that $\tau$ is not path-like. Then there is a red edge in $S^*(\tau)$
  whose endpoints belong to two distinct \GB-components. By
  \cref{lem:fusionPathExistence}, there exists a square-tiled surface $\tau'$
  obtained from $\tau$ by a sequence of horizontal cylinder shears so that this
  red edge is part of a fusion path that ends in leaves of $\Gamma(\tau)$.  By
  \cref{lem:fusionPathDualGraph} performing a single vertical cylinder shear on
  this fusion path results in a square-tiled surface $\tau''$ whose weighted
  stable graph $\Gamma(\tau'')$. One can continue this process until it reaches
  a path-like square-tiled surface.

  Let us now count how many cylinder shears are performed.  Each step involves
  some power of horizontal cylinder shears and a single vertical cylinder
  shear. We first derive a bound on the number of steps performed. Each step
  reduces the weighted stable graph by contracting a path between two
  leaves. The number of steps is hence bounded by the number of leaves, which is
  $O(k)$ (because each leaf of $\Gamma(\tau)$ corresponds to a component
  containing a half-edge). In particular, this bounds the number of vertical
  cylinder shears by $O(k)$.  Let us now bound the number of powers of
  horizontal cylinder shears.  At each step, this number is bounded by the
  length of the image of the fusion path in $\Gamma(\tau)$ used at that
  step. But since the fusion paths from different steps pass through different
  edges of the weighted stable graph their total number is at most the number of
  edges in $\Gamma(\tau)$ which also is $O(k)$.
\end{proof}

\section{Reconfiguring path-like configurations}
\label{ssec:reconfigurationPathLike}

We now show that any two path-like configurations can be reconfigured efficiently
into one another. Namely we prove the following.

\begin{proposition} \label{lem:movesBetweenPathLike} Let $(\mu,k)$ be a spherical
  profile with $\mu_1 \le 1$. Then any pair of path-like square-tiled surfaces
  $\tau$ and $\tau'$ in $\ST_{quad}(\mu)$ can be connected with a sequence of
  $O(k)$ powers of cylinder shears.
\end{proposition}

Note that a closer look at the proof shows that this sequence uses \emph{powers} of
vertical cylinder shears instead of vertical cylinder shears only in the
strata of the form $([1, 2^{\mu_{2}},3],4)$.

Throughout this section, all the square-tiled surfaces we consider will be
spherical and path-like. Using \cref{thm:dualTree}, we work directly on
decorated plane trees on which we will define reconfiguration operations called
\emph{glue and cut} (see \cref{ssec:glueAndCut}) and \emph{decoration exchange}
(see \cref{ssec:decorationExchange}).

To show that two decorated plane trees are equivalent, a common technique in
reconfiguration consists in showing that all decorated plane trees are
equivalent to a specific one. The sketch of the proof is as follows. First, we
show that all decorated plane trees are equivalent to some \emph{nice} decorated
plane tree: a decorated plane tree $(T,L)$ with at most one $\overline L$-leaf
and when it exists, this leaf is attached to a vertex of maximal degree (see
\cref{ssec:nice}). Second, we show that all nice decorated plane trees are
equivalent up to a sequence of glue and cut operations to some nice decorated
plane tree $(T,L)$ with the additional property that $T$ is a caterpillar,
i.e. a tree whose internal vertices form a path (see
\cref{ssec:triangleFreeCaterpillar}). Finally, we show that all nice decorated
plane caterpillars are equivalent by reordering the vertices of degree at least
three (see \cref{ssec:triangleFreeCaterpillar}). Using these intermediate
lemmas, we conclude the proofs of \cref{lem:movesBetweenPathLike} and
\cref{thm:connectednessSphere} in \cref{ssec:movesBetweenPathLike}.

\subsection{Glue and cut}\label{ssec:glueAndCut}
Let us recall from \cref{ssec:pathLike} that path-like configurations are
square-tiled surfaces that admit a single $\{G,B\}$-component which furthermore
is a path. To a path-like configuration one associates a decorated plane tree
$(T,L)$ which encodes its orbit under the unique $\{G,B\}$ cylinder shear, see
\cref{thm:dualTree}.

The glue and cut operation on a decorated plane tree $(T,L)$ is a
reconfiguration operation that produces a new decorated plane tree $(T',L')$. It
is better described decomposed in two steps. First, we delete two $L$-leaves and
replace them by an edge between their parents. Then we delete an edge $uv$ in
the newly formed cycle and replace it by two pending $L'$-leaves attached to $u$
and $v$ (see \cref{fig:glueCut}). The resulting decorated plane tree is
$(T',L')$. Note that $|L| = |L'|$.

\begin{figure}[!ht]
  \begin{center}%

    \begin{subfigure}[t]{.48\textwidth}
      \centering
      \includegraphics[width=.95\textwidth]{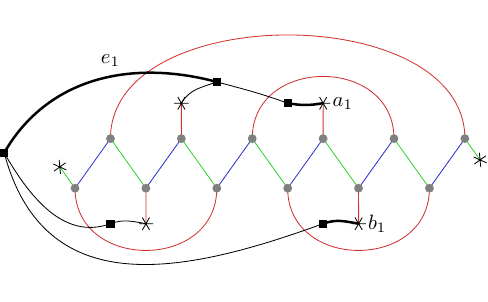}%
      \caption{Up to a power of a horizontal shear, the two half-edges $a_1$ and
        $b_1$ impacted by the gluing are incident to a common green edge.}
      \label{fig:glueCutA}
    \end{subfigure}
    \hfill
    \begin{subfigure}[t]{.48\textwidth}
      \centering
      \includegraphics[width=.95\textwidth]{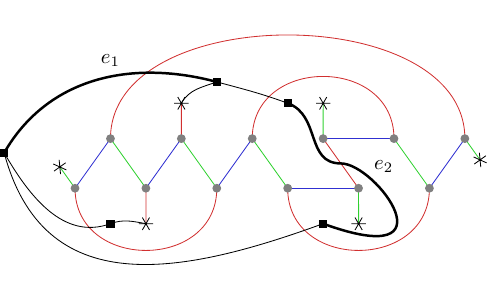}%
      \caption{The configuration obtained from \cref{fig:glueCutA} after the
        vertical shear on the \RG-path connecting the dual of $a_1$ and $b_1$.}
      \label{fig:glueCutB}
    \end{subfigure}

    \begin{subfigure}[t]{.48\textwidth}
      \centering
      \includegraphics[width=.95\textwidth]{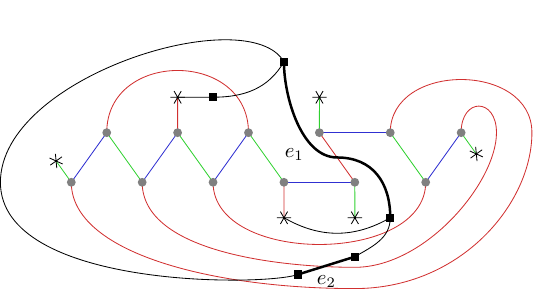}%
      \caption{The configuration obtained from \cref{fig:glueCutB} after powers
        of horizontal shears on its two \GB-paths.}
       \label{fig:glueCutC}
    \end{subfigure}
    \hfill
    \begin{subfigure}[t]{.48\textwidth}
      \centering
      \includegraphics[width=.95\textwidth]{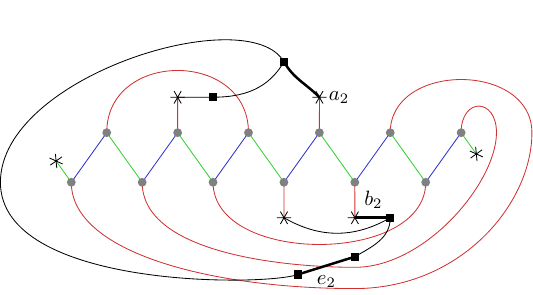}%
      \caption{The configuration obtained from \cref{fig:glueCutC} after the
        vertical shear on the \RG-path containing the dual of $e_1$.}
       \label{fig:glueCutD}
     \end{subfigure}

  \end{center}
  \caption{Glue and cut operation. The outer face is marked by a square
    node. The edges and leaves impacted by the glue and cut operation are
    thickened.\label{fig:glueCut}}
\end{figure}

The following lemma shows that glue and cut operations on decorated plane
trees decompose as a sequence of cylinder shears on the underlying
square-tiled surfaces.
\begin{lemma}\label{lem:glueCutCompo}
  Let $\tau_1$ and $\tau_2$ be the two path-like square-tiled surfaces with
  associated decorated plane trees $(T_1,L_1)$ and $(T_2,L_2)$ that differ by one glue
  and cut operation. Then $\tau_1$ can be connected to $\tau_2$ with a sequence
  of shears made of at most two vertical shears and four powers of horizontal
  shears.
\end{lemma}
\begin{proof}
  The reconfiguration sequence is illustrated by an example in
  \cref{fig:glueCut}.

  Recall that $L_1$ is the subset of leaves of $T_1$ which are dual to red
  half-edges in $\tau_1$. Let $a_1$ and $b_1$ denote the elements of $L_1$ that
  are joined by the glue operation. Let also $e_1$ denote the edge of $T_1$
  removed by the cut operation.

  First, perform in $\tau_1$ a power of the horizontal shear on the unique
  \GB-path until the dual half-edges of $a_1$ and $b_1$ are at distance 3 in the
  tricoloured cubic graph $S^*(\tau)$ and separated by a green edge (see
  \cref{fig:glueCutA}). Then, performing a vertical shear on the path between
  the dual half-edges to $a$ and $b$ breaks the \GB-path into two \GB-paths $P$
  and $Q$ (see \cref{fig:glueCutB}). After this operation, the tricoloured cubic
  graph is not path-like anymore and the dual of the red edges is a graph that
  consists of the tree $T_1$ in which $a_1$ and $b_1$ have been joined to form
  an edge (see \cref{fig:glueCutB}). We call $e_2$ the edge in this graph
  resulting from the fusion of $a_1$ and $b_1$. Note that the edge $e_1$ still
  exists in this graph and is distinct from $e_2$.

  Next, by performing a power of a horizontal shear on each of the \GB-paths $P$
  and $Q$, we move the red edge dual to $e_1$ such that it connects the
  endpoints of $P$ and $Q$ and that there is a green half-edge adjacent at each
  of these endpoints (see \cref{fig:glueCutC}). The choice of which endpoints of
  $P$ and $Q$ are used at this step is irrelevant.  Next, we perform a vertical
  shear on the \RG-path containing the red edge dual to $e_1$. This operation
  destroys $e_1$ in the graph dual to the red edges and creates instead to
  leaves $a_2$ and $b_2$. The two paths $P$ and $Q$ of the tricoloured cubic
  graph are merged into a single \GB-path, thereby obtaining a path-like
  configuration again (see \cref{fig:glueCutD}). By construction, the decorated
  plane tree associated to this resulting tricoloured cubic graph is
  $(T_2, L_2)$.  Finally, by \cref{thm:dualTree} we can use one more power of a
  \GB-shear along the unique \GB-path to reach $\tau_2$.

  The sequence of cylinder shears constructed above from $\tau_1$ to $\tau_2$
  provides the bound from the statement and concludes the proof.
\end{proof}

\subsection{Decoration exchange} \label{ssec:decorationExchange}
In this subsection, we present another reconfiguration operation on decorated
plane trees called \emph{decoration exchange}, that given a decorated plane tree
$(T,L)$, modifies $L$ but not $T$. Similarly to the glue and cut operation from
\cref{ssec:glueAndCut} we prove that this operation on decorated plane trees can
also be obtained by a sequence of cylinder shears on the underlying square-tiled
surfaces.

Let $(T,L)$ be a decorated plane tree containing the two following subtrees: a
$\overline L$-leaf $u$, and two $L$-leaves $v, w$ adjacent to a common vertex
$x$ and consecutive in the cyclic ordering of the vertices around $x$. Let $L'$
be the set of leaves of $T$ obtained by replacing $v$ by $u$ in $L$. We say that
$(T,L)$ and $(T,L')$ differ by a \emph{decoration exchange}.

The operation is illustrated in \cref{fig:decorationExchangeDef}.
\begin{figure}[!ht]

\begin{subfigure}[t]{.48\linewidth}
\centering
\includegraphics{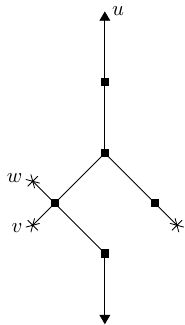}
\end{subfigure}
\hfill
\begin{subfigure}[t]{.48\linewidth}
\includegraphics{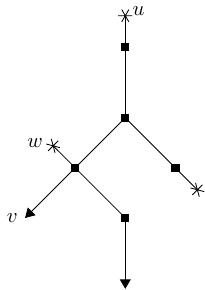}
\end{subfigure}
\caption{A decoration exchange on decorated plane trees.}
\label{fig:decorationExchangeDef}
\end{figure}

\begin{lemma}\label{lem:transferCompo}
  Let $(T,L)$ and $(T, L')$ be two decorated plane trees that differ by a
  decoration exchange. Let $\tau$ and $\tau'$ be two path-like configurations
  whose decorated plane trees are respectively $(T,L)$ and $(T,L')$
  respectively. Then $\tau$ can be connected to $\tau'$ by a sequence made of at
  most two vertical cylinder shear and 4 powers of horizontal shears.
\end{lemma}

  \begin{figure}[h!]

    \begin{subfigure}[t]{.48\linewidth}
      \centering
      \includegraphics{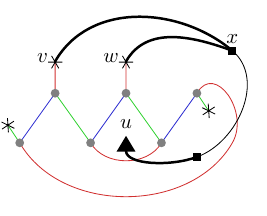}
      \caption{Up to a power of a horizontal shear, the half-edge corresponding
        to $w$ is adjacent to the $1$-face in the initial configuration}
      \label{fig:triangleTransferA}
    \end{subfigure}
    \hfill
    \begin{subfigure}[t]{.48\linewidth}
      \centering
      \includegraphics{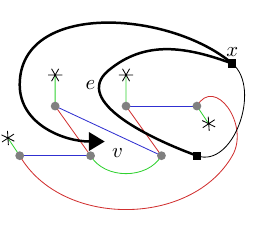}
      \caption{After a vertical shear, the $1$-face is moved and two half-edges
        are merged into an edge $e$}
      \label{fig:triangleTransferB}
    \end{subfigure}

    \begin{subfigure}[t]{.48\linewidth}
      \centering
      \includegraphics{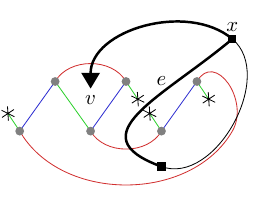}
      \caption{Same configuration as \cref{fig:triangleTransferB}}
      \label{fig:triangleTransferC}
    \end{subfigure}
    \hfill
    \begin{subfigure}[t]{.48\linewidth}
      \centering
      \includegraphics{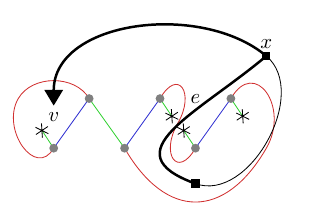}
      \caption{After a power of a horizontal shear, the edge $e$ connect the
        extremities of the two \GB-paths $P_1$, $P_2$}
      \label{fig:triangleTransferD}
    \end{subfigure}

    \begin{subfigure}[t]{.48\linewidth}
      \centering
      \includegraphics{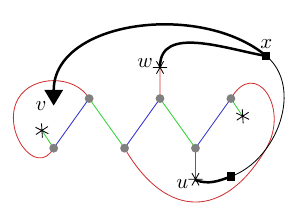}
      \caption{After a vertical shear on $e$ and the incident red half-edges, $e$
        is replaced by two $L'$-leaves}
      \label{fig:triangleTransferE}
    \end{subfigure}

    \caption{How to realise a decoration exchange with a sequence of cylinder
      shears}\label{fig:decorationExchangeProof}
  \end{figure}

\begin{proof}
  \cref{fig:decorationExchangeProof} shows an illustrative example of the
  tricoloured cubic graphs in the reconfiguration sequence.

  Recall that on a decorated plane tree $(T,L)$ associated to a path-like
  configuration $\tau$, the {$\overline L$-leaves} correspond to the 1-faces of
  the tricoloured cubic graph $S^*(\tau)$, which are geometrically triangles.
  Equivalently, they correspond to vertices of degree 3 in the square-tiled
  surface $S(\tau)$.

  Let $u$ be the $\overline L$-leaf of $T$ and $v, w$ the two $L$-leaves
  impacted by the decoration exchange. Let $x$ be the common neighbour of $v$
  and $w$. Perform in $S^*(\tau)$ a power of the horizontal shear on the unique
  \GB-path until the half-edge dual to $w$ is adjacent to the face of degree 3
  dual to $u$ (see \cref{fig:triangleTransferA}). Perform a vertical shear on
  the \RG-path composed of the half-edges dual to $v$ and $w$ and passing by the
  face of degree 3 dual to $u$. This separates the \GB-path into two \GB-paths
  $P$ and $Q$, adds $v$ to $L$ and replaces $u$ and $w$ by an edge $e$ in the
  decorated plane tree (see
  \cref{fig:triangleTransferB,fig:triangleTransferC}). Perform on $P$ and $Q$ a
  power of a horizontal shear, until the edge dual to $e$ connects the
  extremities of both paths (see \cref{fig:triangleTransferD}). Perform a
  vertical shear on the edge dual to $e$ and the incident green half-edges. This
  replaces $e$ by two leaves $u$ and $w$ of $L'$ and yields the desired
  decorated plane tree $(T,L')$.

  Finally, by \cref{thm:dualTree} we can use one more power of a \GB-shear along
  the unique \GB-path to reach $\tau'$.
\end{proof}

\subsection{Making a decorated plane tree nice} \label{ssec:nice}
In this section we describe how the glue and cut
operation and the decoration exchange operation can be combined
to transform any decorated plane tree into a nice tree.

Recall the following definition: A decorated plane tree $(T,L)$ of profile
$(\mu,k-2)$ is \emph{nice} if $\mu_1 = 0$ or if $\mu_1 = 1$ and the only
$\overline L$-leaf of $T$ is adjacent to a vertex of maximal degree.

The following technical lemmas connect decorated plane trees to nice ones, but
only in the subset of spherical profiles with $\mu_1 \le
1$. \cref{lem:removingTriangles} handles all such profiles but those of the form
$([1,2^{\mu_2},3], 4)$, which are handled by \cref{lem:removingTriangles2}.
\begin{lemma}\label{lem:removingTriangles}
  Let $(T,L)$ be a decorated plane tree of profile $(\mu,k-2)$, such that
  $\sum \mu_i(i - 2) = k-4$, with $\mu_1\le 1$ and
  $(\mu, k) \neq ([1,2^{\mu_2},3], 4)$. There exists a sequence composed of at
  most one decoration exchanges and two glue and cut operations that leads to a
  nice decorated plane tree of profile $(\mu,k-2)$.
\end{lemma}

\begin{proof}
  If $\mu_1 = 0$, then $(T,L)$ is already nice. If $\mu_1 = 1$ and $k = 3$, then
  $(T,L)$ is also nice because $T$ is a path. Finally, if $\mu_1 =1$ and
  $k \ge 4$, then the assumption $(\mu,k) \neq ([1,2^{\mu_2},3], 4)$ is
  equivalent to $k \ge 5$, that is $|L| \ge 3$. The reconfiguration sequence is
  illustrated by an example in \cref{fig:niceTransfer}.

  Let $u$ be a vertex of $T$ of maximal degree such that $L$ intersects at least
  three connected components of $T \setminus u$. Such a vertex exists if $T$ has
  maximum degree at least four, in which case any vertex of maximum degree
  satisfies this condition because there is only one $\overline L$-leaf of
  $T$. If $T$ has maximum degree three, then it contains at least two vertices
  of vertices of degree three and at most one of them does not verify this
  condition.

  Let $t$ be the $\overline L$-leaf of $T$.
  Let $v,w,x$ be three leaves of $L$ in distinct connected
  components of $T - u$, such that the connected components containing $v$ and
  $w$ appear consecutively around $u$ (see \cref{fig:niceTransferA}). Perform
  the glue and cut operation that replaces $v$ and $w$ by an edge and cuts one
  of the edges incident to $u$ (see \cref{fig:niceTransferA}). One of the two
  leaves $y,z$ created by this operation is adjacent to $u$, say $y$ (see
  \cref{fig:niceTransferB}). Perform the glue and cut operation that replaces
  $x$ and $z$ by an edge and cuts the edge incident to $u$ and consecutive to
  $y$ (see \cref{fig:niceTransferB}). In the decorated plane tree $(T_f,L_f)$
  obtained after this operation (see \cref{fig:niceTransferD}), $u$ is
  adjacent to two consecutive leaves of $L_f$ and a decoration exchange on $t$ and these
  leaves results in a nice decorated plane tree.
  \begin{figure}[h!]
    \centering
    \begin{subfigure}{.48\linewidth}
      \centering
      \includegraphics{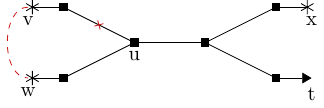}
      \caption{Original decorated plane tree $(T',L')$}\label{fig:niceTransferA}
    \end{subfigure}
    \hfill
    \begin{subfigure}{.48\linewidth}
      \centering
      \includegraphics{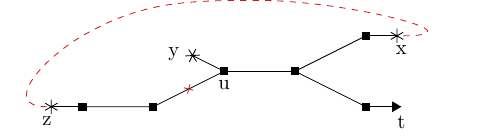}
      \caption{Decorated Plane tree obtained after one glue and cut
        operation}\label{fig:niceTransferB}
    \end{subfigure}

    \begin{subfigure}{.48\linewidth}
      \centering
      \includegraphics{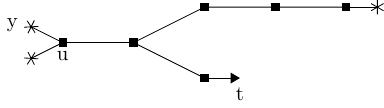}
      \caption{Decorated Plane tree obtained after another glue and cut
        operation}\label{fig:niceTransferC}
    \end{subfigure}
    \hfill
    \begin{subfigure}{.48\linewidth}
      \centering
      \includegraphics{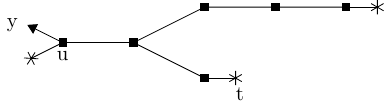}
      \caption{The nice decorated plane tree $(T_f, L_f)$ after a decoration
        exchange}\label{fig:niceTransferD}
    \end{subfigure}
    \caption{Obtaining a nice decorated plane tree when $\mu_1 + k \ge 6$. The
      glue and cut operations resulting in the next subfigure are represented in
      red: the glued half-edges are connected by a dashed path and the cut edge
      is crossed.}
    \label{fig:niceTransfer}
  \end{figure}
\end{proof}

\begin{lemma}\label{lem:removingTriangles2}
  Let $\tau$ be a path-like configuration with profile
  $(\mu, k) = ([1,2^{\mu_2},3], 4)$. There exists a sequence composed of $O(1)$
  powers of horizontal shears and powers of vertical shears that leads to
  path-like configuration corresponding to a nice decorated plane tree.
\end{lemma}

\begin{proof}
  The decorated plane tree $(T,L)$ associated to $\tau$ has one vertex $u$ of
  degree three, $\mu_2$ vertices of degree two, and three leaves, two of them
  belonging to $L$. After one glue and cut operation, we obtain a decorated
  plane tree $(T',L')$ in which an $L'$-leaf is adjacent to $u$. Let $t$ be the
  $\overline L'$-leaf and $m$ be number of vertices of degree two on the path
  between $u$ and $t$. Note that $(T',L')$ is nice if and only if $m=0$, thus
  our goal is to reduce $m$.

  Let $\tau'$ be some path-like configuration corresponding to $(T',L')$. By
  \cref{lem:glueCutCompo}, $\tau$ and $\tau'$ differ by $O(1)$ vertical shears
  and powers of horizontal shears. Label the vertices of $S^*(\tau')$ in the
  order on which they appear on the unique \GB-path. Up to performing a power of
  the horizontal shear, one can assume that $\tau'_R(1) = 3$, i.e. the only
  1-face in $S^*(\tau')$ contains no half-edges in its border but is adjacent to
  a green half-edge via the vertex labelled 1 (see
  \cref{fig:specialStrataEvena,fig:specialStrataOdda}).

  Thus, $(1,2,3, \dots 2m+2,2m+3, \dots 2\mu_2+2)$ is the unique \GB-path in
  $S^*(\tau')$ and $P= (1,3,2,5,4, \dots 2i+3, 2i+2, \dots 2m+3,2m+2)$ is a
  \RG-path in $S^*(\tau')$.

  We claim that after performing $m$ \RG-shears on $P$, we obtain a path-like
  configuration $\tau''$, whose \GB-path is the following sequence. If m is even
  (see \cref{fig:specialStrataEvenb}), then the \GB-path is:
  $$
  (2\mu_2+2, \dots 2m+4, m+3, m+2, (\underbrace{m+1-i, m+2-i,m+4+i,
    m+3+i}_{\text{for } 1 \le i \le m/2}), 1)
  $$
  If $m$ is odd (see \cref{fig:specialStrataOddb}), then the \GB-path is:
  $$
  (2\mu_2+2, \dots 2m+4,(\underbrace{m+i, m+3+i,m+3-i,
    m-i}_{\text{for } 1 \le i \le (m-1)/2}), 2m, 2m+3, 3, 1, 2m+2)
  $$

  \begin{figure}[h!]
    \centering
    \begin{subfigure}{\linewidth}
      \centering
      \includegraphics{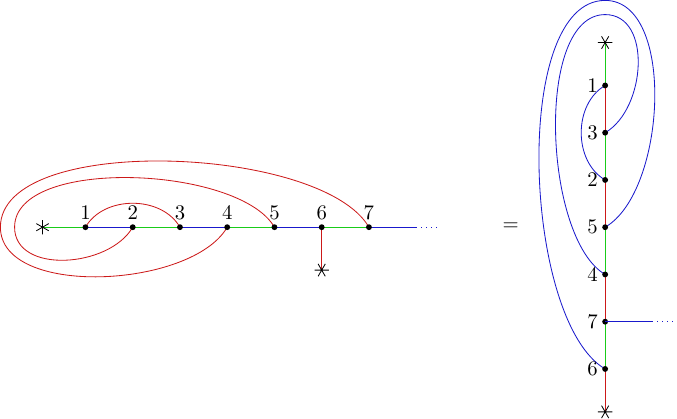}
      \caption{$S^*(\tau')$ is path-like and contains the \RG-path $1,3,2,5,4,
        \dots 2i+3, 2i+2, \dots 2m+3,2m+2$.}
      \label{fig:specialStrataEvena}
    \end{subfigure}

    \begin{subfigure}{\linewidth}
      \centering
      \includegraphics{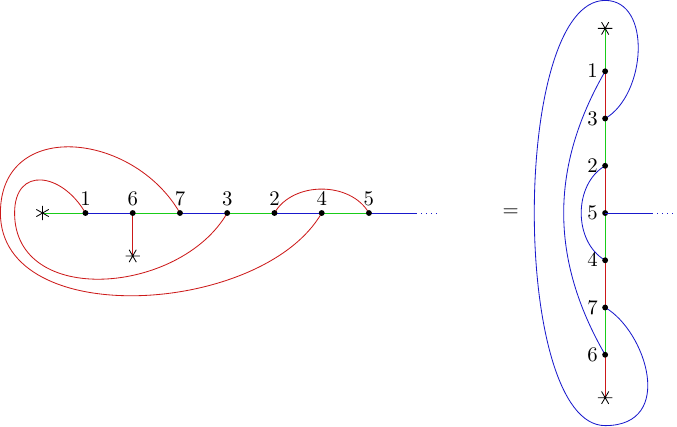}
      \caption{After $m$ \RG-shears, $S^*(\tau'')$ is still path-like and
        corresponds to a nice decorated plane tree}
      \label{fig:specialStrataEvenb}
    \end{subfigure}
    \caption{Obtaining a nice decorated plane tree when $(\mu,k) = ([1,2^{\mu_2},3],
      4)$. Even case, here m= 2.}
    \label{fig:specialStrataEven}
  \end{figure}

  \begin{figure}[h!]
    \centering
    \begin{subfigure}{\linewidth}
      \centering
      \includegraphics{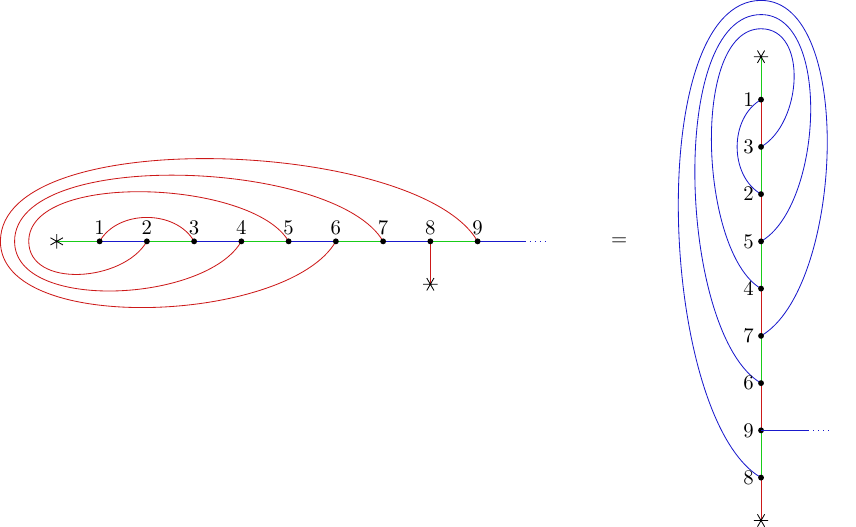}
      \caption{$S^*(\tau')$ is path-like and contains the \RG-path $1,3,2,5,4,
        \dots 2i+3, 2i+2, \dots 2m+3,2m+2$.}
      \label{fig:specialStrataOdda}
    \end{subfigure}

    \begin{subfigure}{\linewidth}
      \centering
      \includegraphics{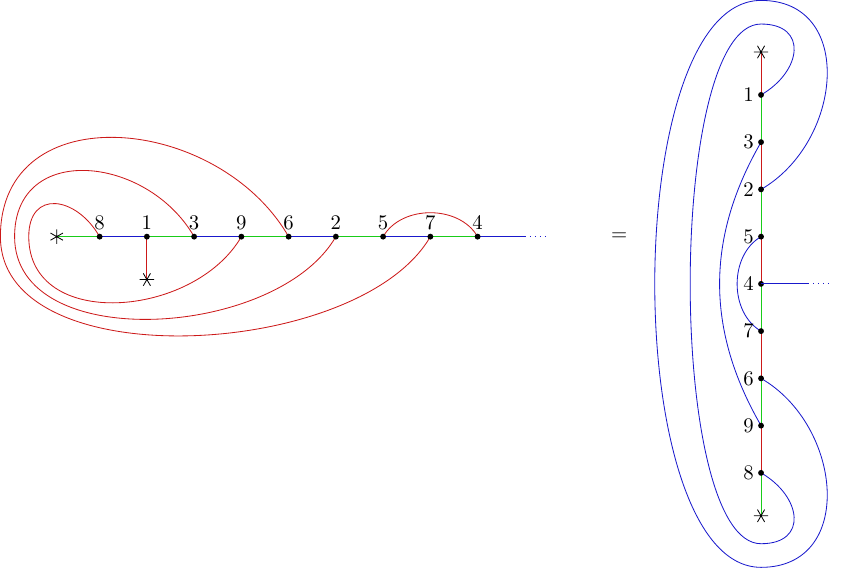}
      \caption{After $m$ \RG-shears, $S^*(\tau'')$ is still path-like and
        corresponds to a nice decorated plane tree}
      \label{fig:specialStrataOddb}
    \end{subfigure}
    \caption{Obtaining a nice decorated plane tree when $(\mu,k) = ([1,2^{\mu_2},3],
      4)$. Odd case, here m= 3.}
    \label{fig:specialStrataOdd}
  \end{figure}

  Each case can be easily checked by an induction on $\lfloor m/2\rfloor$, the
  key argument being that the order in which the vertices $1, \dots 2m+3$ are
  visited by the \GB-path of $S^*(\tau'')$ is monotone with respect to the
  distance to the middle of the \RG-path $P$. The 1-face of $S^*(\tau'')$,
  induced by the vertices $\{m, m+2, m+3\}$ if $m$ is even, respectively by
  $\{m+1, m+2, m+4\}$ if $m$ is odd, is adjacent to the 3-face of $S^*(\tau'')$
  via the red edge $(m \ m+3)$, respectively $(m+1 \ m+2)$. In other words, $t$
  is adjacent to $u$ in the nice decorated plane tree $(T'',L'')$ associated to
  $\tau''$, which concludes the proof.
\end{proof}

By combining \cref{lem:removingTriangles} and \cref{lem:removingTriangles2}
along with \cref{lem:transferCompo} and \cref{lem:glueCutCompo}, we directly
obtain the following.
\begin{corollary}\label{cor:removingTriangles}
  Let $\tau$ be a spherical path-like configuration. There exists a sequence
  composed of $O(1)$ powers of horizontal shears and powers of vertical shears
  that leads to path-like configuration corresponding to a nice decorated plane
  tree.
\end{corollary}

\subsection{Equivalence of nice decorated plane trees}\label{ssec:triangleFreeCaterpillar}
Let $\mu = [1^{\mu_1},2^{\mu_2}, 3^{\mu_3} \ldots]$ and
$k = \sum_i \mu_i (i - 2) + 4$, such that $\mu_1 \le 1$. To prove
\cref{thm:connectednessSphere}, we only need to prove that all nice decorated
plane trees with the profile $(\mu,k-2)$ are equivalent up to a sequence of glue
and cut operations.

We will proceed in two steps: first proving that any nice decorated plane tree
is equivalent up to a sequence of glue and cut operations to a nice decorated
plane tree $(T,L)$ where $T$ is a caterpillar, and then showing that all such
decorated plane trees are equivalent up to a sequence of glue and cut
operations.

\begin{figure}[!ht]
  \begin{center}%
    \includegraphics{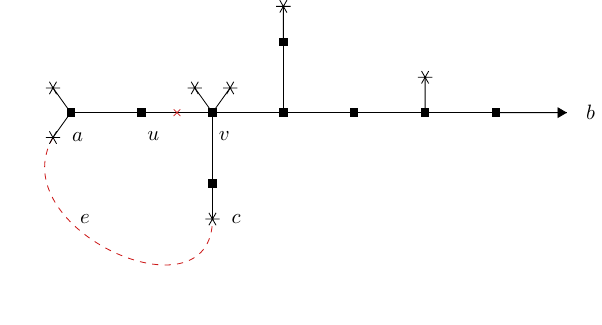}%
\end{center}
\caption{Extending the longest path of a decorated plane tree. The dashed edge
      represents the gluing of $L$-leaves while the edge that is cut is crossed.}
\label{fig:extendCaterpillar}
\end{figure}

\begin{lemma}\label{lem:treeToCaterpillar}
  Let $(T,L)$ be a nice decorated plane tree of profile $(\mu,k-2)$. There
  exists a sequence of at most $k$ glue and cut operations and that connects
  $(T,L)$ to a nice decorated plane tree $(T',L')$ with identical profile, where
  $T'$ is a caterpillar.
\end{lemma}
\begin{proof}
  The assumption $\mu_1 \le 1$ guarantees that there is at most one
  $\overline L$-leaf of $T$. Let $P$ be a maximal path, containing the
  $\overline L$-leaf if it exists. Let $a$ and $b$ be the two leaves at the
  extremities of $P$ (see \cref{fig:extendCaterpillar}), without loss of
  generality, we assume that $a \in L$. We prove that if $T$ is not a
  caterpillar, then we can increase the length of $P$ by performing one glue and
  cut operation.

  Let $c$ be a leaf at distance at least 2 from $P$. Let $Q$ be the path from
  $a$ to $c$ and $uv$ be the last edge of $Q$ belonging to $P$. Perform the glue
  and cut operation that replaces the leaves $a$ and $c$ by an edge $e$ and that
  removes the edge $uv$. It results in a longer path going from $b$ to $v$ via
  $P$, then to $e$ via $Q$ and finally to $u$ via $P$.
\end{proof}

We call \emph{spine} of a caterpillar its internal vertices, in the order they
appear on the path going from one extremity to the other. Let $(T^*,L^*)$ be the
nice decorated plane tree such that $T^*$ is a caterpillar whose spine
$(u_1, \dots u_n)$ is a sequence of non-increasing degree, in which all leaves
adjacent to $u_2, \dots u_n$ belong to $L^*$. Note that this definition leaves
out the ambiguity of the positions of the leaves around each vertex in the plane
embedding of $T^*$. Hence, we further require all $u_i$ to be \emph{cyclically
  ordered}, that is that all the $L^*$-leaves around $u_i$ appear after
$u_{i-1}$ (or $u_2$ for $i=1$) and consecutively in the anti-clockwise order
(see \cref{fig:canonicalCaterpillar}).
\begin{figure}[!ht]
  \centering
  \includegraphics{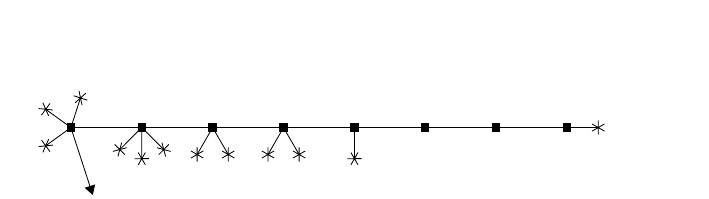}
  \caption{The nice decorated plane tree $(T^*,L^*)$ for the profile $([1, 2^3, 3,4^2,5^2],12)$}
  \label{fig:canonicalCaterpillar}
\end{figure}

\begin{lemma}\label{lem:caterpillarToCanonical}
  Let $(T,L)$ be a nice decorated plane tree with profile $(\mu,k-2)$, such that
  $T$ is a caterpillar. There exists a sequence of $O(k)$ glue and cut
  operations and decoration exchanges leading from $(T,L)$ to $(T^*,L^*)$.
\end{lemma}

\begin{proof}[Proof of \cref{lem:caterpillarToCanonical}]
  Let $(T,L)$ be a nice decorated plane tree of profile $(\mu,k-2)$ with $T$ a
  caterpillar and $(u_1, \dots u_n)$ its spine. If $k+\mu_1=4$, there is only
  one nice decorated plane tree of profile $(\mu,k-2)$, so we are done. We now
  assume that $k+\mu_1 \ge 5$. We first prove a claim that allows us to perform
  some permutations of the spine. We will use it to sort the degree of the
  vertices on the spine and obtain $(T^*,L^*)$.
  \begin{claim}\label{cl:spineSwitch}
    Assume that all $u_1$ is adjacent to at least two $L$-leaves. Let $u_i$ be a
    vertex of the spine. By performing at most three glue and cut operations and
    at most two decoration exchanges, one can move $u_i$ to obtain a nice
    decorated plane caterpillar $(T',L')$ with spine
    $(u_i, u_1, \dots u_{i-1},\allowbreak u_{i+1}, \dots u_n)$. Moreover, $u_i$
    is cyclically ordered in $(T',L')$ and each other vertex $u_j$ stays
    cyclically ordered if it was the case in $(T,L)$ and all adjacent leaves in
    $(T,L)$ belonged to $L$.
  \end{claim}
  \begin{proof}
    Let $I$ be the set of indices for which $u_j$ is cyclically ordered and has
    all its adjacent leaves in $L$. The procedure is illustrated by
    \cref{fig:spineSwitch}.

    Denote $a$ and $b$ the first and last $L$-leaves adjacent to $u_1$ in the
    anti-clockwise order after $u_2$. Note that if $1 \in I$, then all its
    adjacent $L$-leaves appear between $a$ and $b$ in the anti-clockwise
    order. Likewise, for all $j \in I\setminus\{1\}$, all leaves adjacent to
    $u_j$ appear between $u_{j-1}$ and $u_{j+1}$ in the anti-clockwise order.

    Note that $u_n$ is adjacent to some $L$-leaf $c$: by assumption $(T,L)$ is
    nice, so if $\deg(u_n) =2 < \deg(u_1)$, $u_n$ is adjacent to one $L$-leaf;
    and if $\deg(u_n) \ge 3$, then $u_n$ is also adjacent to at least one
    $L$-leaf (see \cref{fig:spineSwitchA}). Take $c$ to be the first $L$-leaf
    after $u_{n-1}$ in the clockwise order.

    Perform the glue and cut operation that replaces the leaves $b$ and $c$ by
    an edge and that cuts the edge $u_{i-1}u_i$ to replace it by a leaf $d$
    adjacent to $u_{i-1}$ and a leaf $e$ adjacent to $u_i$ (see
    \cref{fig:spineSwitchA}). It produces a caterpillar with spine
    $(u_{i}, \dots u_n, u_1 \dots u_{i-1})$ (see
    \cref{fig:spineSwitchB}). Perform a glue and cut operation that replaces the
    leaves $a$ and $e$ by an edge and that cuts the edge $u_iu_{i+1}$ to replace
    it by a leaf $f$ adjacent to $u_i$ and a leaf $g$ adjacent to $u_{i+1}$ (see
    \cref{fig:spineSwitchB}). It produces a tree that is not a caterpillar
    anymore, in which $u_n$ and $u_i$ are adjacent to $u_1$ (see
    \cref{fig:spineSwitchC}). Finally, perform the glue and cut operation that
    replaces the leaves $c$ and $g$ by an edge and cuts the edge $u_1u_n$ (see
    \cref{fig:spineSwitchC}). It produces a caterpillar with the desired spine
    (see \cref{fig:spineSwitchD}).

    As it can be easily checked on \cref{fig:spineSwitch}, all vetices $u_j$
    with $j \in I$ are still cyclically ordered. It is however possible that
    $u_1$ is not yet cyclically ordered if it is adjacent to the
    $\overline{L}$-leaf, in which case $u_1$ is also adjacent to an $L$-leaf and
    two consecutive decoration exchanges with any other $L$-leaf can make $u_1$
    cyclically ordered, as desired.
    \begin{figure}[!ht]
      \begin{center}%
        \begin{subfigure}{\linewidth}
          \centering
          \includegraphics{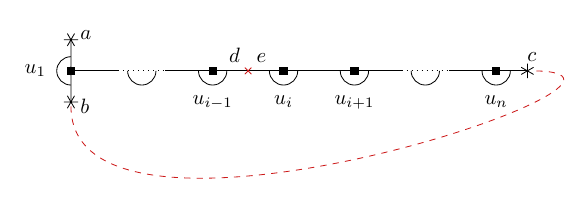}
          \caption{Original caterpillar with spine $(u_1, \dots u_n)$ and $u_1$
            adjacent to at least two $L$-leaves}\label{fig:spineSwitchA}
        \end{subfigure}
        \begin{subfigure}{\linewidth}
          \centering
          \includegraphics{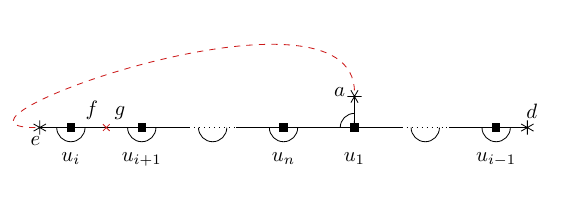}\\
          \caption{After one glue and cut operation, the spine becomes
            $(u_i,u_{i+1} \dots u_n, u_1, \dots u_{i-1})$ }\label{fig:spineSwitchB}
        \end{subfigure}
        \begin{subfigure}{\linewidth}
          \centering
          \includegraphics{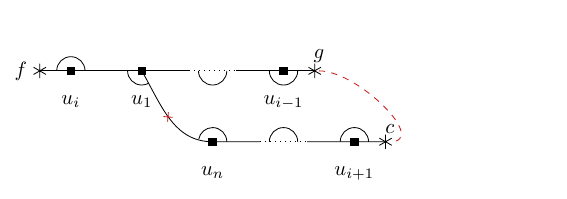}\\
          \caption{After one more glue and cut operation, the decorated plane
            tree is not a caterpillar anymore}\label{fig:spineSwitchC}
        \end{subfigure}
        \begin{subfigure}{\linewidth}
          \centering
          \includegraphics{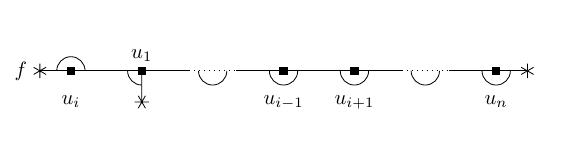}\\
          \caption{Final caterpillar with spine $(u_i, u_1, \dots u_{i-1},
            u_{i+1},\dots u_n)$}\label{fig:spineSwitchD}
        \end{subfigure}
      \end{center}%
      \caption{Moving $u_i$ to the front of the caterpillar. Again, the glue and
        cut operations resulting in the next subfigure are represented in red:
        the glued half-edges are connected by a dashed path and the cut edge is
        crossed. The arcs around the vertices represent the positions of the
        remaining $L$-leaves, for the vertices $u_j$ with $j\in I$.}
      \label{fig:spineSwitch}
    \end{figure}
  \end{proof}

  We first ensure that the one of the extremities of the spin is of degree at
  least three. If that is not the case, each extremity is adjacent to an
  $L$-leaf, say $a$ and $b$ respectively. Since $k + \mu_1\ge 5$, there is a
  internal vertex $u$ that is adjacent to an $L$-leaf. Perform the glue and cut
  operation that replaces $a$ and $b$ by an edge and cuts an edge incident to
  $u$. It produces a caterpillar in which $u$ is adjacent to two $L$-leaves and
  is placed at one extremity of the spin.

  Let $v_1, \dots v_p$ be all the vertices of degree at least three in $T$,
  ordered such that for all $i$, $\deg(v_i) \le \deg(v_{i+1 })$ and for all $i <
  p$, the leaves adjacent to $v_i$ all belong to $L$.

  We can apply \cref{cl:spineSwitch} successively to each of the $v_i$. This
  results in $(T^*,L^*)$. The corresponding reconfiguration sequence has length
  $O(\sum_i \mu_i) = O(k)$
\end{proof}

\subsection{Proof of \cref{lem:movesBetweenPathLike} and
  \cref{thm:connectednessSphere}}
\label{ssec:movesBetweenPathLike}
\cref{thm:connectednessSphere} from the introduction follows directly from
\cref{lem:movesToPathLike} that was proved in \cref{sec:connectingtoPathLike}
and \cref{lem:movesBetweenPathLike} that we prove now.

\begin{proof}[Proof of \cref{lem:movesBetweenPathLike}]
  Let $\tau$ be a path-like square-tiled spherical surface with profile
  $(\mu,k)$. By \cref{cor:removingTriangles}, $\tau$ is connected by a sequence
  of $O(1)$ powers of cylinder shears to a path-like square-tiled surface
  $\tau'$ that corresponds to a nice decorated plane tree $(T', L')$.

  Let $(T^*,L^*)$ be the decorated plane tree of profile $(\mu,k-2)$ defined
  before \cref{lem:caterpillarToCanonical}. By \cref{thm:dualTree}, there is a
  path-like configuration $\tau^*$ with profile $(\mu,k)$ that maps to
  $(T^*,L^*)$.

  By \cref{lem:treeToCaterpillar,lem:caterpillarToCanonical}, $(T',L')$ can be
  connected to $(T^*,L^*)$ via $O(k)$ glue and cut operations and decoration
  exchanges. Hence, $\tau$ and $\tau^*$ are equivalent up to $O(k)$ powers of
  cylinder shears by \cref{lem:glueCutCompo} and \cref{lem:transferCompo}. This
  is true for all path-like square-tiled surfaces in $ST_{quad}(\mu,k)$, which
  concludes the proof of \cref{lem:movesBetweenPathLike}.
\end{proof}

Combined with \cref{lem:movesToPathLike}, this proves that any two spherical
square-tiled surfaces of profile $(\mu,k)$ with $\mu_1 \le 1$ are connected up
to $O(k)$ powers of cylinder shears, thereby concluding the proof of
\cref{thm:connectednessSphere}.

\section{Hyperelliptic square-tiled surfaces}
\label{sec:hyperellipticSquareTiledSurfaces}

This section focuses on the hyperelliptic Abelian components. We first give in
\cref{ssec:hyperellepticComponents} a more precise description of the quotient
of such square-tiled surfaces by the hyperelliptic involution. Then in
\cref{ssec:halfShears} we explain the effect of a cylinder shear on the quotient
by the hyperelliptic involution. In \cref{ssec:hyperellipticPathLike} we show
how hyperelliptic Abelian square-tiled surfaces can be connected to path-like
configurations using cylinder shears. Finally, we conclude the proof of
\cref{thm:connectednessHyperelliptic} in \cref{ssec:connectednessHyperelliptic}
by reducing to \cref{thm:connectednessSphere}.

\subsection{Hyperelliptic components}\label{ssec:hyperellepticComponents}
Following~\cite{KontsevichZorich2003} we define the subsets
$\ST_{Ab}^{hyp}([2^{\mu_2}, (2g)^2])$ and $\ST_{Ab}^{hyp}([2^{\mu_2}, 4g-2])$ of
respectively $\ST_{Ab}([2^{\mu_2}, (2g)^2])$ and $\ST_{Ab}([2^{\mu_2}, 4g-2])$
(recall from \cref{ssec:introSquareTiledSurfaces} that we use the notation
$\ST(\mu)$ instead of $\ST(\mu,k)$ when $k = 0$, which is always the case for
Abelian square-tiled surfaces). These subsets correspond to the square-tiled
surfaces that belong to the so-called hyperelliptic components of the moduli
space of Abelian differentials. By
\cref{prop:connectedComponentsAndConnectedComponents} proven at the end
of~\cref{ssec:AbelianAndQuadraticDifferentials}, these subsets are invariant
under cylinder shears. However, we propose in this subsection an alternative
combinatorial proof.

Let $\tau = (\tau_R, \tau_G, \tau_B) \in (S_n)^3$ be a square-tiled surface in
some stratum $\ST(\mu, k)$. An \emph{automorphism} of
$\tau$ is a permutation $\alpha \in S_n$ that normalises simultaneously $\tau_R$,
$\tau_G$ and $\tau_B$ that is : $\alpha \tau_i \alpha^{-1} = \tau_i$ for
any colour $i \in \{R,G,B\}$. We denote $\Aut(\tau)$ the automorphism group of $\tau$.

Let $H$ be a subgroup of $\Aut(\tau)$. We define the quotient square-tiled surface
$\tau' = \tau / H$ whose set of squares are the orbits of $H$ on $[n]$
and $(\tau_R', \tau_G', \tau_B')$ are the induced action on the orbits.
Note that when taking the quotient, one needs to choose a labelling of
the $H$-orbits in $[n]$ with $[n']$. A canonical
choice consists in ordering these orbits according to the minimal element.

A square-tiled surface $\tau$ of genus $g \ge 2$ is \emph{hyperelliptic} if it
admits an automorphism $\alpha$ which is of order 2 and such that the quotient
$\tau / \langle \alpha \rangle$ has genus 0.

Let $g \ge 2$ and $\mu_2 \ge 0$.
We define $\ST_{Ab}^{hyp}([2^{\mu_2}, 4g-2])$ to be the subset of
$\ST_{Ab}([2^{\mu_2}, 4g-2])$ that are hyperelliptic and
$\ST_{Ab}^{hyp}([2^{\mu_2}, (2g)^2])$ to be the subset of
$\ST_{Ab}([2^{\mu_2}, (2g)^2])$ that are hyperelliptic and such
that the hyperelliptic involution exchanges the two singularities
of degree $2g$.

It is a standard result that a translation or half-translation surface of genus
$g \ge 2$ admits at most one hyperelliptic involution. Furthermore, if it exists
it is central in the automorphism group (i.e. all other automorphisms commute
with the hyperelliptic involution). These two facts follow from the result that
for a Riemann surface $X$, the induced action on homology gives rise to an
injective map $\Aut(X) \to \Sp(H_1(X; \bZ))$ and the image of the hyperelliptic
involution is $-Id$, see~\cite[Theorem~6.8]{FarbMargalit2011}.

We warn the reader that for the profile $\mu=[2^{\mu_2}, (2g)^2]$ the condition that
the hyperelliptic involution exchanges the two singularities of degree $2g$ is
important. More precisely, there exists hyperelliptic square-tiled surfaces
in $\ST_{Ab}([2^{\mu_2}, (2g)^2])$ whose quotient belongs to a spherical
square-tiled surface in some strata $\ST_{quad}([1^{\mu'_1},\allowbreak
2^{\mu'_2}, (g)^2], k)$, see \cref{fig:exampleHyperellipticNoExchange}.
\begin{figure}[!ht]
  \begin{subfigure}{.55\textwidth}
    \centering \includegraphics{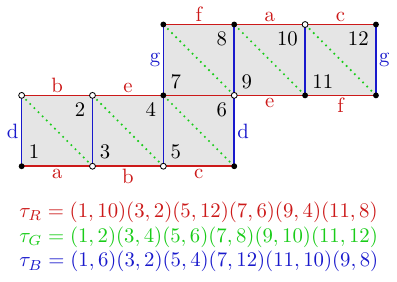}
    \caption{A surface in $\ST_{Ab}([6^2])$ with hyperelliptic involution
      $\alpha=(1,4)(2,3)(5,6)(7,12)(8,11)(9,10)$.}
    \label{sfig:hyp66}
  \end{subfigure}
  \hspace{.05\textwidth}
  \begin{subfigure}{.4\textwidth}
    \centering \includegraphics{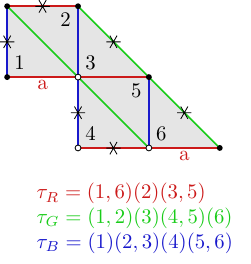}
    \caption{The quotient of \cref{sfig:hyp66} in $\ST_{quad}([3^2], 6)$.}
  \end{subfigure}
  \caption{A hyperelliptic square-tiled surface in $\ST_{Ab}([6^2])$ not in
    $\ST_{Ab}^{hyp}([6^2])$: the quotient belongs to $\ST_{quad}([3^2], 6)$. }
  \label{fig:exampleHyperellipticNoExchange}
\end{figure}

\begin{lemma}\label{lem:hyperellipticAbelianQuotient}
  For $g \ge 2$ and $\mu_2 \ge 0$, the quotient by the hyperelliptic
  involution induces the bijections
  \begin{equation}\label{eq:hyperellipticQuotient1}
    \ST^{hyp}_{Ab}([2^{\mu_2}, 4g-2])
    \to
    \bigcup_{\substack{k' + \mu'_1 = 2g+1 \\ \mu'_1 + 2 \mu'_2 = \mu_2}}
    \ST_{quad}([1^{\mu'_1}, 2^{\mu'_2}, 2g-1], k').
  \end{equation}
  and
  \begin{equation}\label{eq:hyperellipticQuotient2}
    \ST^{hyp}_{Ab}([2^{\mu_2}, (2g)^2])
    \to
    \bigcup_{\substack{k' + \mu'_1 = 2g + 2\\ \mu'_1 + 2 \mu'_2 = \mu_2}}
    \ST_{quad}([1^{\mu'_1}, 2^{\mu'_2}, 2g], k').
  \end{equation}
\end{lemma}

\begin{proof}
  In the first case, the hyperelliptic involution preserves the singularity of
  degree $4g-2$. The image of this singularity in the quotient by the
  hyperelliptic involution is a singularity of degree $2g-1$. By Euler's
  characteristic computation, one obtains that $k' + \mu'_1 = 2g + 1$.

  In the second case, by definition, the hyperelliptic involution exchanges the
  two singularities of degree $2g$. The image of this pair of singularities in
  the quotient by the hyperelliptic involution is a singularity of degree
  $2g$. Again, by Euler's characteristic computation, one obtains that
  $k' + \mu'_1 = 2g+2$.

  This proves that the quotient by the hyperelliptic involution indeed belongs to
  the right hand sides of~\eqref{eq:hyperellipticQuotient1}
  and~\eqref{eq:hyperellipticQuotient2} respectively.

  We now briefly explain why these maps are bijections. This follows from the
  standard construction that to any quadratic square-tiled surface, one can
  associate a unique Abelian square-tiled surface and an involution such that
  the quotient gives back the quadratic differential.
\end{proof}

The following lemma is a consequence of a result of Kontsevich and
Zorich~\cite[Section~2.1]{KontsevichZorich2003}. Alternatively, a direct
combinatorial proof follows immediately from
\cref{lem:shearQuotientinHyperelliptic}, stated in the next subsection after
defining half cylinder shears.

\begin{lemma}
  Let $\mu=[2^{\mu_2}, 4g-2]$ or $\mu=[2^{\mu_2}, (2g)^2]$ for some $g \ge 2$ and $\mu_2 \ge 0$.
  Then $\ST^{hyp}_{Ab}(\mu)$ is preserved by cylinder shears.
\end{lemma}

\begin{remark}\label{rem:quadratic}
  Let us mention that for quadratic square-tiled surfaces, there also exist hyperelliptic
  connected components, see~\cite{Lanneau2008}. They correspond exactly to quadratic
  square-tiled surfaces in genus $g \ge 1$ that are hyperelliptic and whose quotient
  belongs to $\ST_{quad}(1^{\mu_1}, 2^{\mu_2}, a, b)$ where $a, b \ge 2$. That is, instead
  of having a single singularity $d \ge 3$ as in \cref{lem:hyperellipticAbelianQuotient}
  on the sphere, there are two.

  However, we were not able to apply the techniques we develop in this article
  in order to show the connectedness of $\ST^{hyp}_{quad}(\mu)$.
\end{remark}

\subsection{Half cylinder shears}\label{ssec:halfShears}
At first glance, \cref{lem:hyperellipticAbelianQuotient} and
\cref{conj:squareTiledSurfacesConnectedComponents} seem in
contradiction. Namely, from~\eqref{eq:hyperellipticQuotient1}, the quotient by
the hyperelliptic involution maps the subset of square-tiled surfaces
$\ST^{hyp}_{Ab}([2^{\mu_2}, 4g-2])$ to a disjoint union of
$\ST_{quad}([1^{\mu'_1}, 2^{\mu'_2}, 2g-1], k')$. And
\cref{conj:squareTiledSurfacesConnectedComponents} asserts that the left side
of~\eqref{eq:hyperellipticQuotient1} is connected under cylinder shears, while
the right side is obviously not by \cref{lem:profileIsInvariant}. The reason is
that if $\tau$ is a square-tiled surface in $\ST^{hyp}_{Ab}([2^{\mu_2}, 4g-2])$
and $\tau'$ its quotient by the hyperelliptic involution then a cylinder shear
in $\tau$ does not necessarily correspond to a cylinder shear in $\tau'$. One
sometimes obtains what we call a \emph{half cylinder shear} that modifies the
profile and which we study now.

Let $\tau$ be a square-tiled surface and $c$ a \GB-cycle that separates the
square-tiled surface such that one of the connected components of the complement
of $c$ contains only hexagons, $1$-faces and red half-edges. In more a
combinatorial way, let $c = c_1 \sqcup c_2$ the decomposition into two
$\tau_B \circ \tau_G$-orbits. We ask that $c_1$ is stable under $\tau_R$ and
that the faces bounded by $c$ together with the red edges with ends in $c_1$ are
only hexagons and either two half-edges or two $1$-faces. The horizontal
\emph{half cylinder shear} along $c$ is obtained by changing $\tau_R$ along
$c_1$ as follows (see also \cref{fig:halfShear})
\[
\tau'_R(i) =
\left\{
\begin{array}{ll}
\tau_G \circ \tau_B \circ \tau_R (i) & \text{if $i \in c_1$} \\
\tau_R(i) & \text{otherwise.}
\end{array}
\right.
\]
We similarly define vertical half cylinder shears on \RG-cycles.

\begin{figure}[h]
  \begin{center}
    \includegraphics{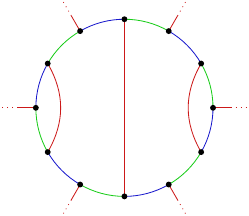}
    \hfill
    \includegraphics{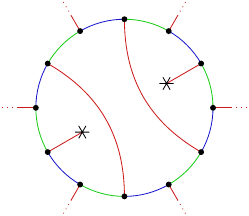}
    \hfill
    \includegraphics{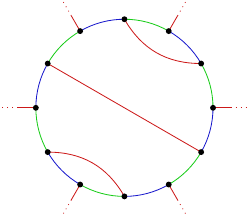}
  \end{center}
  \caption{A cylinder shear as the composition of two half cylinder shears.}
  \label{fig:halfShear}
\end{figure}
Similarly to the case of the cylinder shears, half cylinder shears induce a
natural bijection on the vertices of $S^*(\tau)$, that we will not define here.

\begin{lemma} \label{lem:shearQuotientinHyperelliptic} Let $\tau$ be a
  square-tiled surface of genus $g \ge 2$ in either
  $\ST^{hyp}_{Ab}([2^{\mu_2}, 4g-2])$ or $\ST^{hyp}_{Ab}([2^{\mu_2},
  (2g)^2])$. Let $\tau'$ be its quotient by the hyperelliptic involution.  Let
  $c$ be a \GB-cycle in $\tau$ and $c'$ its image in $\tau'$.

  If $c'$ belongs to a horizontal cylinder bounded by a \GB-path $c''$
  (equivalently if the height $h$ in the edge decoration $(w,h)$ of the weighted
  stable graph $\Gamma(\tau')$ is a half integer) then the quotient of
  $\Shear_{c,B,G}(\tau)$ by the hyperelliptic involution is isomorphic to
  $\Shear_{c'',B,G}(\tau')$.

  If $c'$ belongs to a horizontal cylinder bounded by a \GB-cycle $c''$
  isolating a component of $S^*(\tau')$ made only of hexagons and either red
  half-edges or 1-faces, then the quotient by the hyperelliptic involution of
  $\Shear_{c,B,G}(\tau)$ is isomorphic to the horizontal half cylinder shear
  along $c''$ in $\tau'$.
\end{lemma}

\begin{proof}
  We only sketch the proof, which is straightforward.

  Each cylinder of $S(\tau)$ is preserved by the hyperelliptic involution. The
  first case corresponds to the situation where the height of the cylinder
  containing the curve $c$ has odd height. In that case, the circumference in
  the middle of the cylinder is mapped to a component of the critical graph that
  corresponds to a \GB-path in $S^*(\tau)$. Performing a cylinder shear in $c$
  is the same as performing a cylinder shear along that \GB-path.

  In the second case, when the height is even, the circumference is mapped to a
  component of the critical graph made of red edges.  In that case, the cylinder
  shear in $c$ becomes a horizontal half cylinder shear in the quotient.
\end{proof}

The lower bound of \cref{thm:connectednessHyperelliptic} is a direct consequence
of \cref{lem:shearQuotientinHyperelliptic}. Its proof is very similar to that of
\cref{prop:squareTiledLowerBound} for spherical profiles, by considering half
cylinder shear on top of cylinder shears.

\begin{proposition}\label{prop:hyperellipticLowerBound}
  Let $\mu \in \{[2^{\mu_2},(2g)^2],[2^{\mu_2},4g-2]\}$, with $g>1$. Then there
  are square-tiled surfaces in $\ST^{hyp}_{Ab}(\mu)$ that are separated by
  $\Omega(g)$ cylinder shears.
\end{proposition}
\begin{proof}
  Let $\mu_1'$ and $\mu_2'$ such that $\mu_1' + 2\mu_2' = \mu_2$ and
  $\mu_1'\le 1$. Let $k' = 2g+1 - \mu_1'$ if $\mu$ is of the form
  $[2^{\mu_2},4g-2]$ and $k' = 2g+2 - \mu_1'$ otherwise. Note that $(\mu',k')$
  is a spherical profile and that $k' \ge 2g$. By \cref{thm:pathLikeExistence},
  there is a path-like square-tiled surface $\tau$ of profile $(\mu',k')$, and
  by symmetry, there is also a square tiled-surface $\sigma$ of profile
  $(\mu',k')$ that has a single vertical cylinder which furthermore is a path.

  A shear can only change by at most two the number of half-edges coloured red in
  $S^*(\tau)$. Likewise, a half cylinder shear changes the number of red
  half-edges by at most two. Since $\tau$ and $\sigma$ have respectively $k'-2$
  and at most two red half-edges. By \cref{lem:hyperellipticAbelianQuotient} and
  \cref{lem:shearQuotientinHyperelliptic}, the corresponding square-tiled
  surfaces of $\ST^{hyp}_{Ab}(\mu)$ are separated by at least $g-2$ shears,
  which concludes the proof.
\end{proof}
\subsection{Connecting to a path-like configuration}\label{ssec:hyperellipticPathLike}
We continue to work on spherical square-tiled surface.  We now describe how to
circumvent the restriction $\mu_1 \le 1$ on the profile when connecting to a
path-like configuration as in \cref{sec:connectingtoPathLike} or when connecting
two path-like configurations as in \cref{ssec:reconfigurationPathLike}.  The
procedure uses the half cylinder shears introduced just before.

The analogue of \cref{lem:movesToPathLike} we prove in that context is the
following.

\begin{lemma} \label{lem:hyperellipticToPathLike} Let $\tau$ be a spherical
  square-tiled surface with profile $([1^{\mu_1}, 2^{\mu_2}, d], k)$ where
  $d \ge 3$.  Then $\tau$ can be connected with a sequence of $O(k + \mu_1)$
  powers of cylinder shears and $O(\mu_1)$ half cylinder shears to a path-like
  square-tiled surface.
\end{lemma}

\begin{proof}
  Let $\tau$ be a spherical square-tiled surface with profile
  $([1^{\mu_1}, 2^{\mu_2}, d], k)$. The weighted stable tree $\Gamma(\tau)$ is a
  star whose center is the unique vertex $v_0$ whose decoration $\mu^{v_0}$
  contains the element $d$ of the profile. The other vertices, that necessarily
  are leaves, have decorations either equal to $([2^m], 2)$, or $([1, 2^m], 1)$
  or $([1^2,2^m], 0)$ where $m \ge 0$. For each leaf with decoration
  $([1^2,2^m], 0)$, we use a half cylinder shear to turn it in to a leaf with
  decoration $([2^{m+1}], 2)$.  We obtain a square-tiled surface satisfying the
  assumptions of \cref{lem:fusionPathExistence} on which there hence exists a
  fusion path. Performing a vertical cylinder shear on that fusion path reduces
  the degree of the special vertex $v_0$ in $\Gamma(\tau)$.

  Now it remains to treat the case when the degree of $v_0$ is one. In that
  case, $\Gamma(\tau)$ has a single other vertex $v_1$ with decoration either
  $([2^m], 2)$, or $([1,2^m], 1)$ or $([1^2,2^m], 0)$ with $m \ge 0$. If $v_0$
  is such that its decoration satisfies $k^{(v_0)} > 0$, then up to performing a
  horizontal half cylinder shear in $v_1$, one can also use
  \cref{lem:fusionPathExistence} to find a fusion path. We now further assume
  that $k{(v_0)} = 0$. Up to a horizontal half cylinder shear, one can assume
  that $\mu_1^{(v_1)} > 0$. It is easy to see that there exists a \RG-cycle $c$
  that ``connects'' a 1-face in $v_0$ to a 1-face in $v_1$ in the following
  sense: $c$ crosses every \GB-components exactly twice and one of the connected
  component of the complement of $c$ consists only of hexagons and two
  1-faces. Performing a vertical half cylinder shear in $c$ results in a path
  like configuration (see \cref{fig:fusionCycle}).
  \begin{figure}[!ht]
    \centering
    \includegraphics{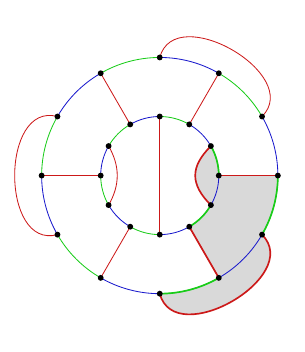}
    \hspace{1cm}
    \includegraphics{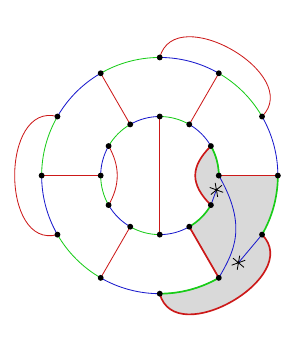}
    \caption{A \RG-cycle in $S^*(\tau)$ as used in the proof of
      \cref{lem:hyperellipticToPathLike}. The shaded region is made of two
      1-faces and one 2-face (or hexagon).}
    \label{fig:fusionCycle}
  \end{figure}
\end{proof}

\subsection{Reducing to the spherical case: proof of
  \cref{thm:connectednessHyperelliptic}}\label{ssec:connectednessHyperelliptic}

We focus on the upper bound of \cref{thm:connectednessHyperelliptic}, as the
lower bound was proven by \cref{prop:hyperellipticLowerBound}.

\begin{proof}[Proof of \cref{thm:connectednessHyperelliptic}]
  Let us consider a square-tiled surface in either
  $\ST_{Ab}^{hyp}([2^{\mu_2},\allowbreak (2g)^2])$ or
  $\ST_{Ab}^{hyp}([2^{\mu_2}, 4g-2])$ for some $\mu_2 \ge 0$.  By
  \cref{lem:hyperellipticAbelianQuotient}, its quotient by the hyperelliptic
  involution belongs to respectively either a stratum
  $\ST_{quad}([1^{\mu'_1}, 2^{\mu'_2}, 2g-1], k')$ or
  $\ST_{quad}([1^{\mu'_1}, 2^{\mu'_2}, 2g], k')$ for some
  $\mu'_1, \mu'_2, k' \ge 0$.  By \cref{lem:shearQuotientinHyperelliptic}, the
  cylinder shears in $\tau$ correspond to cylinder shears or half cylinder
  shears in the quotients. Hence, it is enough to connect
  $\displaystyle \bigcup_{\substack{k' + \mu'_1 = 2g+1 \\ \mu'_1 + 2 \mu'_2 =
      \mu_2}} \ST_{quad}([1^{\mu'_1}, 2^{\mu'_2},\allowbreak 2g-1], k')$ or
  $\displaystyle \bigcup_{\substack{k' + \mu'_1 = 2g + 2\\ \mu'_1 + 2 \mu'_2 =
      \mu_2}} \ST_{quad}([1^{\mu'_1}, 2^{\mu'_2}, 2g], k')$ using both cylinder
  shears and half cylinder shears.

  By \cref{thm:connectednessSphere} we know that cylinder shears are enough to
  connect the spherical strata when $\mu'_1\le1$. Let us first note that the
  parity of $\mu'_1$ is fixed by the profile $(\mu,k)$ by
  \cref{lem:parityCondition}.  Hence, depending on $(\mu,k)$ only one of the
  quotient stratum with $\mu'_1 \le 1$ is non-empty. What remains to do is to
  connect all spherical square-tiled surface to these base cases.

  To do so, we first apply \cref{lem:hyperellipticToPathLike} that allows us to
  reach a path-like configuration $S^*(\tau)$.

  Next, we explain how to remove pairs of 1-faces in $S^*(\tau)$ until there
  remains at most one of them (the corresponding sequence is illustrated by an
  example in \cref{fig:triangleEdition}). Let $A$ and $B$ be two 1-faces of
  $S^*(\tau)$ (see \cref{fig:triangleEditionA}). Perform in $\tau$ a power of
  the horizontal shear on the unique \GB-path until $A$ and $B$ share a blue
  edge $e$ (see \cref{fig:triangleEditionB}). Let $f$ be the red edge in the
  boundary of $A$. Perform a vertical half cylinder shear on the \RG-cycle
  bounding $A \cup B$, this replaces $e$ by two blue half-edges and thus cuts
  the \GB-path in two \GB-paths $P_1$ and $P_2$. This also changes the profile
  by trading two $1$-faces for two half-edges and a hexagon (see
  \cref{fig:triangleEditionC}). Perform a horizontal cylinder shear on $P_1$ and
  $P_2$ such that each of them has a extremity incident to $f$ and a green
  half-edge (see \cref{fig:triangleEditionD}). Perform a vertical cylinder shear
  on $f$ and the incident green half-edges to obtain again a single \GB-path.
  \begin{figure}[h!]
    \begin{subfigure}[t]{.48\linewidth}
      \centering
      \includegraphics{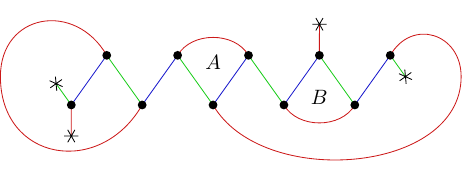}
      \caption{Starting configuration.}
      \label{fig:triangleEditionA}
    \end{subfigure}
    \hfill
    \begin{subfigure}[t]{.48\linewidth}
      \centering
      \includegraphics{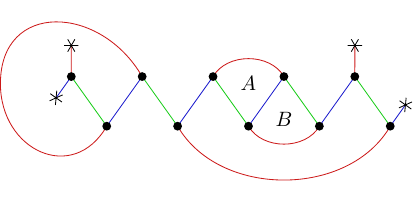}
      \caption{After a power of a horizontal shear the two $1$-faces $A$ and $B$
        are adjacent and separated by a blue edge $e$.}
      \label{fig:triangleEditionB}
    \end{subfigure}

    \begin{subfigure}[t]{.48\linewidth}
      \centering
      \includegraphics{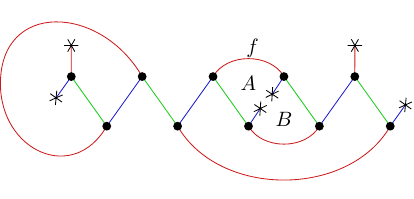}
      \caption{After a vertical half cylinder shear, $e$ is replaced by two blue
        half-edges. This impacts the profile.}
      \label{fig:triangleEditionC}
    \end{subfigure}
    \hfill
    \begin{subfigure}[t]{.48\linewidth}
      \centering
      \includegraphics{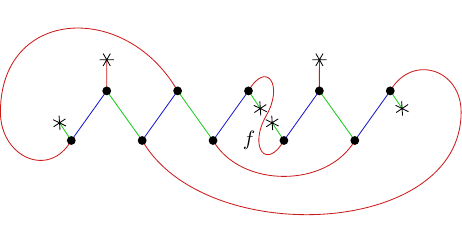}
      \caption{After two horizontal shears, a red edge $f$ originally on the
        boundary of $A$ is now adjacent to the extremities of both \GB-paths.}
      \label{fig:triangleEditionD}
    \end{subfigure}

    \begin{subfigure}[t]{.48\linewidth}
      \centering
      \includegraphics{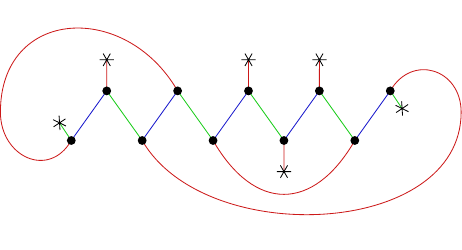}
      \caption{Final configuration, obtained by performing the
        vertical shear containing $f$.}
      \label{fig:triangleEditionE}
    \end{subfigure}
    \caption{Removing pairs of 1-faces in the path-like configuration.}
    \label{fig:triangleEdition}
  \end{figure}

  The resulting square-tiled surface is path-like and has two fewer 1-faces. By
  repeating this procedure inductively, one obtains a spherical path-like
  square-tiled surface $\tau'$ with at most one 1-face. In other words, the
  profile $(\mu',k')$ of $\tau'$ is spherical and has $\mu_1' \le 1$.  We
  reduced the general case to the two known base cases of
  \cref{thm:connectednessSphere} in $O(k + \mu_1)$ powers of cylinder shears and
  half cylinder shears. This concludes the proof of the theorem.
\end{proof}

\section{Discussion}\label{sec:discussion}
\subsection{Equivalence via cylinder shears}
The first open problem we consider is naturally
\cref{conj:squareTiledSurfacesConnectedComponents} in full generality. In this
article, we confirmed \cref{conj:squareTiledSurfacesConnectedComponents} in two
special cases. First, in the Abelian hyperelliptic components (see
\cref{thm:connectednessHyperelliptic}). Second, on the sphere (see
\cref{thm:connectednessSphere}), in all strata when authorizing half-shears, but
only in the strata such that $\mu_1 \le 1$ and $(\mu,k) \neq ([1,2^*,3], 4)$
when restricting to regular shears. This suggest several milestones towards
proving \cref{conj:squareTiledSurfacesConnectedComponents}, that isolate
different points of failure of our proof. We present them ranked by increasing
(supposed) difficulty:

The first question is whether the quadratic hyperelliptic components are also
connected by cylinder shears. Our method to connect to a path-like configuration
is specific to $\mu_1 \le 1$ when using only cylinder shears, and to the strata
with $\sum_{i \ge 3} \mu_i \le 1$ when authorizing half-shears, which does not
cover quadratic hyperelliptic components.

Finally, we conjecture the following, which is a special case of
\cref{conj:squareTiledSurfacesConnectedComponents} and generalises
\cref{thm:connectednessSphere} to spherical profiles with $\mu_1 \ge 2$:
\begin{conjecture}\label{conj:connectednessShpere}
  Let $(\mu,k)$ be a spherical profile. The set of square-tiled surfaces
  $\ST(\mu,k)$ is connected by cylinder shears.
\end{conjecture}

\subsection{Further remarks}
\label{ssec:furtherRemarks}
There is an action of $\SL_2(\bZ)$ on square-tiled surfaces that is a natural
restriction of an $\SL_2(\bR)$-action on Abelian and quadratic differentials.
This action can easily be rephrased in terms of cylinder shears: the action of
$\begin{pmatrix}1&1\\0&1\end{pmatrix}$ and
$\begin{pmatrix}1&0\\1&1\end{pmatrix}$ corresponds to simultaneous cylinder
shears in respectively all horizontal and vertical cylinders. Since these two
matrices generate $\SL_2(\bZ)$, connected components of the $\SL_2(\bZ)$
reconfiguration graph are finer than for cylinder shears.  The classification of
these connected components, even conjecturally, is open in general. Beyond the
elementary case of the stratum $\cH(\emptyset)$ on the torus, only partial
results are known in genus 2 and in the so-called Prym loci,
see~\cite{McMullen2005, HubertLelievre2006, LanneauNguyen2020, Duryev}.

\bigskip

Beyond its interest as a delicate reconfiguration problem, it would be
interesting to investigate more the geometric significance of cylinder shears on
square-tiled surfaces. It is well-known that the cardinality of $\ST(\mu, k)$ is
related to the so-called Masur--Veech measure of strata of Abelian and quadratic
differentials~\cite{Zorich2002}. Though the link between the geometry (e.g. the
diameter) of the cylinder shear reconfiguration graph and connected components
of strata remains to explore. In particular, it seems plausible that the mixing
rate of the cylinder shear dynamics is related to the spectral gap of the
$\SL_2(\bR)$-action on the connected components of strata of Abelian and
quadratic strata established in~\cite{AvilaGouezelYoccoz2006, AvilaGouezel2013}.

\bigskip

Finally, cylinders shears on square-tiled surfaces present surprising
similarities with a reconfiguration operation on the colourings of a graph. Two
reconfiguration operations on colourings have been considered: single vertex
recolouring that consists in changing the colour of a single vertex and Kempe
changes that are more complex but are crucial in the proof of the celebrated
Four Colour Theorem. Like Kempe changes, cylinder shears can affect an
arbitrarily large part of the configuration. This non-locality complicates the
analysis of the reconfiguration graph, in particular, most of the arguments to
prove rapid mixing of the associated Markov chains do not seem to apply to such
chains. As we observed at the end of \cref{ssec:shearDef}, there are deeper
connections between cylinder shears and Kempe changes on edge colourings, which
suggest that developments of methods to analyse the mixing of one of these
chains may transfer to the other.

\bibliographystyle{alpha}
\bibliography{flips}

\newcommand{\etalchar}[1]{$^{#1}$}
\begin{thebibliography}{BLdM{\etalchar{+}}22}

\bibitem[AG13]{AvilaGouezel2013}
Artur Avila and S{\'e}bastien Gou{\"e}zel.
\newblock Small eigenvalues of the {Laplacian} for algebraic measures in moduli
  space, and mixing properties of the {Teichm{\"u}ller} flow.
\newblock {\em Ann. Math. (2)}, 178(2):385--442, 2013.

\bibitem[AGY06]{AvilaGouezelYoccoz2006}
Artur Avila, S{\'e}bastien Gou{\"e}zel, and Jean-Christophe Yoccoz.
\newblock Exponential mixing for the {Teichm{\"u}ller} flow.
\newblock {\em Publ. Math., Inst. Hautes {\'E}tud. Sci.}, 104:143--211, 2006.

\bibitem[Ald94]{Aldous1994}
David Aldous.
\newblock Triangulating the circle, at random.
\newblock {\em Am. Math. Mon.}, 101(3):223--233, 1994.

\bibitem[AM24]{AthreyaMasur2024}
Jayadev~S. Athreya and Howard Masur.
\newblock {\em Translation surfaces (to appear)}.
\newblock Providence, RI: American Mathematical Society (AMS), 2024.

\bibitem[BLdM{\etalchar{+}}22]{BuchinEtAl}
Maike Buchin, Anna Lubiw, Arnaud de~Mesmay, Saul Schleimer, and Florestan
  Brunck.
\newblock {Computation and Reconfiguration in Low-Dimensional Topological
  Spaces (Dagstuhl Seminar 22062)}.
\newblock {\em Dagstuhl Reports}, 12(2):17--66, 2022.

\bibitem[Bud17]{Budzinski2017}
Thomas Budzinski.
\newblock On the mixing time of the flip walk on triangulations of the sphere.
\newblock {\em C. R., Math., Acad. Sci. Paris}, 355(4):464--471, 2017.

\bibitem[CFZ11]{CassaigneFerencziZamboni2011}
Julien Cassaigne, S{\'e}bastien Ferenczi, and Luca~Q. Zamboni.
\newblock Combinatorial trees arising in the study of interval exchange
  transformations.
\newblock {\em Eur. J. Comb.}, 32(8):1428--1444, 2011.

\bibitem[CHK{\etalchar{+}}18]{CardinalHoffmanKustersTothWettstein2018}
Jean Cardinal, Michael Hoffmann, Vincent Kusters, Csaba~D. T{\'o}th, and Manuel
  Wettstein.
\newblock Arc diagrams, flip distances, and {Hamiltonian} triangulations.
\newblock {\em Comput. Geom.}, 68:206--225, 2018.

\bibitem[CS20]{CaraceniStauffer2020}
Alessandra Caraceni and Alexandre Stauffer.
\newblock Polynomial mixing time of edge flips on quadrangulations.
\newblock {\em Probab. Theory Relat. Fields}, 176(1-2):35--76, 2020.

\bibitem[DGZ{\etalchar{+}}20]{DelecroixGoujardZografZorich2020}
Vincent Delecroix, {\'E}lise Goujard, Peter Zograf, Anton Zorich, and Engel.
\newblock Contribution of one-cylinder square-tiled surfaces to {Masur}-{Veech}
  volumes.
\newblock In {\em Some aspects of the theory of dynamical systems: a tribute to
  Jean-Christophe Yoccoz. Volume I}, pages 223--274. Paris: Soci{\'e}t{\'e}
  Math{\'e}matique de France (SMF), 2020.

\bibitem[DGZZ20]{delecroix2020Enumeration}
Vincent Delecroix, Elise Goujard, Peter Zograf, and Anton Zorich.
\newblock Enumeration of meanders and {{Masur}}--{{Veech}} volumes.
\newblock In {\em Forum of Mathematics, Pi}, volume~8, page~e4. Cambridge
  University Press, 2020.

\bibitem[DGZZ21]{DelecroixGoujardZografZorich2021}
Vincent Delecroix, {\'E}lise Goujard, Peter Zograf, and Anton Zorich.
\newblock Masur-{Veech} volumes, frequencies of simple closed geodesics, and
  intersection numbers of moduli spaces of curves.
\newblock {\em Duke Math. J.}, 170(12):2633--2718, 2021.

\bibitem[DP19]{DisarloParlier2019}
Valentina Disarlo and Hugo Parlier.
\newblock The geometry of flip graphs and mapping class groups.
\newblock {\em Trans. Am. Math. Soc.}, 372(6):3809--3844, 2019.

\bibitem[Dur]{Duryev}
Eduard Duryev.
\newblock Teichm\"uller curves in genus two: Square-tiled surfaces and modular
  curve.
\newblock arXiv:1905.09312 [math.GT].

\bibitem[EF23]{eppstein2023Improved}
David Eppstein and Daniel Frishberg.
\newblock Improved mixing for the convex polygon triangulation flip walk, 2023.

\bibitem[FM11]{FarbMargalit2011}
Benson Farb and Dan Margalit.
\newblock {\em A primer on mapping class groups}, volume~49 of {\em Princeton
  Math. Ser.}
\newblock Princeton, NJ: Princeton University Press, 2011.

\bibitem[Fra17]{Frati2017}
Fabrizio Frati.
\newblock A lower bound on the diameter of the flip graph.
\newblock {\em Electron. J. Comb.}, 24(1):research paper p1.43, 6, 2017.

\bibitem[FZ10]{FerencziZamboni2010}
S{\'e}bastien Ferenczi and Luca~Q. Zamboni.
\newblock Structure of {{\(K\)}}-interval exchange transformations: induction,
  trajectories, and distance theorems.
\newblock {\em J. Anal. Math.}, 112:289--328, 2010.

\bibitem[HL06]{HubertLelievre2006}
Pascal Hubert and Samuel Leli{\`e}vre.
\newblock Prime arithmetic {Teichm{\"u}ller} discs in {{\({\mathcal H}(2)\)}}.
\newblock {\em Isr. J. Math.}, 151:281--321, 2006.

\bibitem[KZ03]{KontsevichZorich2003}
Maxim Kontsevich and Anton Zorich.
\newblock Connected components of the moduli spaces of {Abelian} differentials
  with prescribed singularities.
\newblock {\em Invent. Math.}, 153(3):631--678, 2003.

\bibitem[Lan08]{Lanneau2008}
Erwan Lanneau.
\newblock Connected components of the strata of the moduli spaces of quadratic
  differentials.
\newblock {\em Ann. Sci. {\'E}c. Norm. Sup{\'e}r. (4)}, 41(1):1--56, 2008.

\bibitem[LN20]{LanneauNguyen2020}
Erwan Lanneau and Duc-Manh Nguyen.
\newblock Weierstrass {Prym} eigenforms in genus four.
\newblock {\em J. Inst. Math. Jussieu}, 19(6):2045--2085, 2020.

\bibitem[LP17]{levin2017markov}
David~A Levin and Yuval Peres.
\newblock {\em Markov chains and mixing times}, volume 107.
\newblock American Mathematical Soc., 2017.

\bibitem[McM05]{McMullen2005}
Curtis~T. McMullen.
\newblock Teichm{\"u}ller curves in genus two: {Discriminant} and spin.
\newblock {\em Math. Ann.}, 333(1):87--130, 2005.

\bibitem[MN19]{mynhardt2019reconfiguration}
CM~Mynhardt and S~Nasserasr.
\newblock Reconfiguration of colourings and dominating sets in graphs.
\newblock In {\em 50 Years of Combinatorics, Graph Theory, and Computing},
  pages 171--191. Chapman and Hall/CRC, 2019.

\bibitem[Mos88]{Mosher1988}
Lee Mosher.
\newblock Tiling the projective foliation space of a punctured surface.
\newblock {\em Trans. Am. Math. Soc.}, 306(1):1--70, 1988.

\bibitem[MRS99]{molloy1997Mixing}
Michael Molloy, Bruce Reed, and William Steiger.
\newblock On the mixing rate of the triangulation walk.
\newblock In {\em Randomization methods in algorithm design. DIMACS workshop,
  Princeton Univ., NJ, USA, December 12--14, 1997}, pages 179--190. Providence,
  RI: AMS, American Mathematical Society, 1999.

\bibitem[Nis18]{Nishimura2018}
Naomi Nishimura.
\newblock Introduction to reconfiguration.
\newblock {\em Algorithms}, 11(4):52, 2018.

\bibitem[Pou14]{Pournin2014}
Lionel Pournin.
\newblock The diameter of associahedra.
\newblock {\em Adv. Math.}, 259:13--42, 2014.

\bibitem[PP17]{ParlierPournin2017}
Hugo Parlier and Lionel Pournin.
\newblock Flip-graph moduli spaces of filling surfaces.
\newblock {\em J. Eur. Math. Soc. (JEMS)}, 19(9):2697--2737, 2017.

\bibitem[PP18a]{ParlierPournin2018a}
Hugo Parlier and Lionel Pournin.
\newblock Modular flip-graphs of one-holed surfaces.
\newblock {\em Eur. J. Comb.}, 67:158--173, 2018.

\bibitem[PP18b]{ParlierPournin2018b}
Hugo Parlier and Lionel Pournin.
\newblock Once punctured disks, non-convex polygons, and pointihedra.
\newblock {\em Ann. Comb.}, 22(3):619--640, 2018.

\bibitem[STT88]{SleatorTarjanThurston-1988}
Daniel~D. Sleator, Robert~E. Tarjan, and William~P. Thurston.
\newblock Rotation distance, triangulations, and hyperbolic geometry.
\newblock {\em J. Am. Math. Soc.}, 1(3):647--681, 1988.

\bibitem[STT92]{SleatorTarjanThurston-1992}
Daniel~D. Sleator, Robert~E. Tarjan, and William~P. Thurston.
\newblock Short encodings of evolving structures.
\newblock {\em SIAM J. Discrete Math.}, 5(3):428--450, 1992.

\bibitem[vdH13]{van2013complexity}
Jan van~den Heuvel.
\newblock The complexity of change.
\newblock {\em Surveys in combinatorics}, 409(2013):127--160, 2013.

\bibitem[Zor02]{Zorich2002}
Anton Zorich.
\newblock Square tiled surfaces and {Teichm{\"u}ller} volumes of the moduli
  spaces of {Abelian} differentials.
\newblock In {\em Rigidity in dynamics and geometry. Contributions from the
  programme Ergodic theory, geometric rigidity and number theory, Isaac Newton
  Institute for the Mathematical Sciences, Cambridge, UK, January 5--July 7,
  2000}, pages 459--471. Berlin: Springer, 2002.

\bibitem[Zor06]{Zorich2006}
Anton Zorich.
\newblock Flat surfaces.
\newblock In {\em Frontiers in number theory, physics, and geometry I. On
  random matrices, zeta functions, and dynamical systems. Papers from the
  meeting, Les Houches, France, March 9--21, 2003}, pages 437--583. Berlin:
  Springer, 2nd printing edition, 2006.

\bibitem[Zor08]{Zorich2008}
Anton Zorich.
\newblock Explicit {Jenkins}-{Strebel} representatives of all strata of
  {Abelian} and quadratic differentials.
\newblock {\em J. Mod. Dyn.}, 2(1):139--185, 2008.

\end{thebibliography}

\end{document}